\documentclass[12pt]{amsart}
\usepackage{amsfonts}
\usepackage{amssymb}
\usepackage{booktabs}
\usepackage{afterpage}
\usepackage[shortlabels]{enumitem}
\usepackage{xcolor}
\usepackage[all]{xy}
\usepackage[margin=1in]{geometry} 
\usepackage{eucal}
\usepackage{hyperref}
\hypersetup{bookmarksdepth=2}
\hypersetup{colorlinks=true}
\hypersetup{linkcolor=blue}
\hypersetup{citecolor=blue}
\hypersetup{urlcolor=blue}

\numberwithin{equation}{section}
\newtheorem{theorem}[equation]{Theorem}
\newtheorem{proposition}[equation]{Proposition}
\newtheorem{lemma}[equation]{Lemma}
\newtheorem{corollary}[equation]{Corollary}
\newtheorem{conjecture}[equation]{Conjecture}

\theoremstyle{definition}
\newtheorem{rmk}[equation]{Remark}
\newenvironment{remark}[1][]{\begin{rmk}[#1] \pushQED{\qed}}{\popQED \end{rmk}}
\newtheorem{eg}[equation]{Example}
\newenvironment{example}[1][]{\begin{eg}[#1] \pushQED{\qed}}{\popQED \end{eg}}
\newtheorem{constr}[equation]{Construction}
\newenvironment{construction}[1][]{\begin{constr}[#1] \pushQED{\qed}}{\popQED \end{constr}}
\newtheorem{defnaux}[equation]{Definition}
\newenvironment{definition}[1][]{\begin{defnaux}[#1]\pushQED{\qed}}{\popQED \end{defnaux}}

\newcommand{\cC}{\mathcal{C}}

\newcommand{\cD}{\mathcal{D}}

\newcommand{\cE}{\mathcal{E}}

\newcommand{\rM}{\mathrm{M}}

\newcommand{\bN}{\mathbf{N}}

\newcommand{\fN}{\mathfrak{N}}

\newcommand{\bP}{\mathbf{P}}

\newcommand{\bS}{\mathbf{S}}

\newcommand{\rf}{\mathrm{f}}

\let\ol\overline
\let\ul\underline

\DeclareMathOperator{\End}{End}

\DeclareMathOperator{\Aut}{Aut}

\DeclareMathOperator{\Hom}{Hom}

\renewcommand{\phi}{\varphi}

\newcommand{\GL}{\mathbf{GL}}

\newcommand{\Th}{\mathfrak{Th}}

\DeclareMathOperator{\Gr}{\mathbf{Gr}}

\newcommand{\ulambda}{\ul{\smash{\lambda}}}

\let\defn\emph
\newcommand{\DOI}[1]{\href{http://doi.org/#1}{\color{purple}{\tiny\tt DOI:#1}}}
\newcommand{\arxiv}[1]{\href{http://arxiv.org/abs/#1}{{\tiny\tt arXiv:#1}}}

\title{Biquadratic spaces of length two}
\date{December 29, 2024}

\author{Alessandro Danelon}
\address{Department of Mathematics, University of Michigan, Ann Arbor, MI, USA}
\email{\href{mailto:adanelon@umich.edu}{adanelon@umich.edu}}
\urladdr{\url{https://public.websites.umich.edu/~adanelon/}}

\author{Andrew Snowden}
\thanks{AS was supported by NSF grant DMS-2301871.}
\address{Department of Mathematics, University of Michigan, Ann Arbor, MI, USA}
\email{\href{mailto:asnowden@umich.edu}{asnowden@umich.edu}}
\urladdr{\url{http://www-personal.umich.edu/~asnowden/}}

\begin{document}

\begin{abstract}
A \defn{tensor space} is a vector space equipped with a finite collection of multilinear forms. The \defn{length} of a tensor space is its length as a representation of its symmetry group. Infinite dimension tensor spaces of finite length are special, highly symmetrical objects. We classify the (universal) biquadratic spaces of length two; there are seven families of them. As a corollary, we establish some cases of the linear analog of the Ryll-Nardzewski theorem. We view this work as a first attempt to classify highly symmetrical tensor spaces.
\end{abstract}	

\maketitle
\tableofcontents

\section{Introduction}

\subsection{Background}

Fix an algebraically closed field $k$ of characteristic~0 and a tuple $\ulambda=[\lambda_1, \ldots, \lambda_r]$ of partitions. A \defn{$\ulambda$-space} is a vector space $V$ equipped with linear functionals $\omega_i \colon \bS_{\lambda_i}(V) \to k$ for each $i$, where $\bS_{\lambda}$ is the Schur functor. Our $\ulambda$-spaces will often be infinite dimensional, but, throughout this paper, we only allow $\ulambda$-spaces of countable dimension. When $\lambda_i=(2)$ for each $i$, we can regard the $\omega_i$'s as symmetric bilinear forms on $V$. We refer to this as the \defn{$r$-quadratic} or \defn{multi-quadratic} case. This paper is chiefly concerned with the \defn{bi-quadratic} case ($r=2$).

Recent work has demonstrated $\ulambda$-spaces to be a rich topic of study. We recall two results. A $\ulambda$-space is called \defn{universal} if every finite dimensional $\ulambda$-space embeds into it. Kazhdan and Ziegler \cite{KaZ} established an important characterization of universal $\ulambda$-spaces (later generalized by Bik, Danelon, Draisma, and Eggermont \cite{BDDE}). In the multi-quadratic case, their theorem says that $V$ is universal if and only if every non-trivial linear combination of the $\omega_i$'s has infinite rank.

A $\ulambda$-space $V$ is called \defn{homogeneous} if whenever $\sigma, \tau \colon W \to V$ are two embeddings of a finite dimensional $\ulambda$-space $W$, there exists an automorphism $g$ of $V$ such that $\sigma = g \circ \tau$. For example, Witt's famous theorem states that a finite dimensional non-degenerate quadratic space is homogeneous. Harman and Snowden \cite{homoten} showed that there exists a universal homogeneous $\ulambda$-space of countable dimension, which is unique up to isomorphism (provided each $\lambda_i$ is non-empty). In particular, there is a distinguished isomorphism class of $\ulambda$-space in countable dimensions, which is rather remarkable as nothing like this is true in finite dimensions. In follow-up work \cite{homoten2}, we constructed some additional ``weakly homogeneous'' spaces.

Classifying $\ulambda$-spaces is an interesting and (in our opinion) important problem. Classifying all $\ulambda$-spaces up to isomorphism seems prohibitively difficult. There are two weaker problems that appear to be more tractable. The first is to classify all $\ulambda$-spaces up to isogeny, which is a weaker equivalence than isomorphism. In \cite{isocubic} we completely solved this problem in the cubic case, and in forthcoming work \cite{isogeny} we treat the general case.

The second problem is to classify certain special $\ulambda$-spaces up to isomorphism. For example, one might consider spaces with large automorphism group (made precise in some way). We note that the (weakly) homogeneous $\ulambda$-spaces constructed in \cite{homoten, homoten2} do indeed have large automorphism groups. This seems to be a much more difficult problem than the previous one, but we are optimistic that, in some form, it is tractable. The purpose of this paper is to solve this problem in the first non-trivial case.

\afterpage{
\clearpage\normalfont

\begin{table} \centering
\caption{The seven families of length two universal biquadratic spaces. In the third column, $d$ denotes an element of $\{1,2,\ldots,\infty\}$, and $p$ and $q$ denote points of $\bP^1$. The final column is the ring of endomorphisms commuting with the automorphism group $G$. See Table~\ref{tab:len2} for how this determines the structure of $V$ as a representation of $G$.} \label{tab:main}
\renewcommand{\arraystretch}{1.2}
\begin{tabular}{@{}llll@{}}
\toprule
Name & Section & Parameters & $\End_G(V)$ \\
\midrule
Ia & \S \ref{ss:Ia} & $p \ne q$ & $k \oplus k$ \\
Ib & \S \ref{ss:Ib} & $d$ and $p$ & $k \oplus k$ \\
Ic & \S \ref{ss:Ic} & $d$ & $k$ \\
IIa & \S \ref{ss:IIa} & $p$ & $k[x]/(x^2)$ \\
IIb & \S \ref{ss:IIb} & $p$ & $k$ \\
IIc & \S \ref{ss:IIc} & $p$ & $k$ \\
III & \S \ref{s:III} & $E \subset k^4$ & $\rM_2(k)$ \\
\bottomrule
\end{tabular}
\end{table}

\begin{table} \centering
\vspace*{4ex}
\caption{Explicit definitions of Type I and II spaces. Here $B$ is an infinite symmetric matrix such that no non-trivial linear combination of its columns has finite support. For Type Ib and Ic, the top left block has size $d$; all other blocks are infinite.} \label{tab:matrix}
\renewcommand{\arraystretch}{1}
\begin{tabular}{@{}lcc@{\hskip 1in}lcc@{}}
\toprule
& $\omega$ & $\omega'$ &
& $\omega$ & $\omega'$ \\
\midrule\addlinespace[6pt]
Ia &
$\begin{pmatrix} 1 & 0 \\ 0 & 0 \end{pmatrix}$ &
$\begin{pmatrix} 0 & 0 \\ 0 & 1 \end{pmatrix}$ &
IIa &
$\begin{pmatrix} 0 & 1 \\ 1 & 0 \end{pmatrix}$ &
$\begin{pmatrix} 0 & 0 \\ 0 & 1 \end{pmatrix}$ \\[12pt]
Ib &
$\begin{pmatrix} 1 & 0 \\ 0 & B \end{pmatrix}$ &
$\begin{pmatrix} 0 & 0 \\ 0 & 1 \end{pmatrix}$ &
IIb &
$\begin{pmatrix} 1 & B \\ B & 0 \end{pmatrix}$ &
$\begin{pmatrix} 0 & 0 \\ 0 & 1 \end{pmatrix}$ \\[12pt]
Ic &
$\begin{pmatrix} 0 & 0 \\ 0 & B \end{pmatrix}$ &
$\begin{pmatrix} 0 & 0 \\ 0 & 1 \end{pmatrix}$ &
IIc &
$\begin{pmatrix} 0 & B \\ B & 0 \end{pmatrix}$ &
$\begin{pmatrix} 0 & 0 \\ 0 & 1 \end{pmatrix}$ \\[12pt]
\bottomrule
\end{tabular}
\end{table}

\begin{table} \centering
\vspace*{4ex}
\caption{Let $V$ be a length two representation of a group $G$ with simple constituents $L_1$ and $L_2$. The following table shows how the structure of $V$ is determined by its endomorphism algbera.} \label{tab:len2}
\renewcommand{\arraystretch}{1.2}
\begin{tabular}{@{}lll@{}}
\toprule
$\End_G(V)$ & Semi-simple? & $L_1 \cong L_2$? \\
\midrule
$\rM_2(k)$ & Yes & Yes \\
$k \oplus k$ & Yes & No \\
$k[x]/(x^2)$ & No & Yes \\
$k$ & No & No \\
\bottomrule
\end{tabular}
\end{table}

\clearpage
}

\subsection{Results}

We now state the main results of this paper. For these, we must assume that our coefficient field $k$ is uncountable. For instance, $k$ could be the field of complex numbers.

Let $V$ be a $\ulambda$-space, and let $G=\Aut(V)$. We can view $V$ as a representation of $G$, and we define the \defn{length} of $V$ to be its length as a representation. ``Finite length'' is one flavor of ``large automorphism group.'' The irreducible (length one) multi-quadratic spaces are fairly easy to classify; see  \S \ref{s:multiq}. The following is our main theorem:

\begin{theorem} \label{mainthm}
The universal biquadratic spaces of length two are those found in the seven families Ia, Ib, Ic, IIa, IIb, IIc, and~III appearing in Table~\ref{tab:main}.
\end{theorem}

Matrices for the Type~I and Type~II spaces are given in Table~\ref{tab:matrix}. The Type~I spaces are exactly those that can be obtained by taking an orthogonal direct sum of two irreducible biquadratic spaces. It would not be difficult to extend the theorem to accommodate non-universal spaces, but this would just make the paper longer without introducing new ideas.

In \cite{homoten}, we introduced the notion of \defn{linearly oligomorphic} group. Roughly speaking, this is a group $G$ acting linearly on a vector space $V$ such that $\Gr_d(V)/G$ is finite dimensional for all $d$, where $\Gr_d$ denotes the Grassmannian; see \S \ref{ss:linolig} for details. We also defined a notion of \defn{linearly $\omega$-categorical} for $\ulambda$-spaces. Essentially, $V$ is linearly $\omega$-categorical if its isomorphism type admits a first-order axiomatization; see \S \ref{ss:logic} for details. As a consequence of Theorem~\ref{mainthm}, we obtain the following:

\begin{theorem} \label{mainthm2}
Let $V$ be a universal biquadratic space of length two. Then:
\begin{enumerate}
\item The group $\Aut(V)$ is linearly oligomorphic.
\item The space $V$ is linearly $\omega$-categorical.
\end{enumerate}
\end{theorem}

The Ryll-Nardzewski theorem is a classical theorem in model theory, which states that a countable relational structure is $\omega$-categorical if and only if its automorphism group is oligomorphic. There is a natural linear analog of this statement (Conjecture~\ref{conj:RN}), which is (in our opinion) one of the important open questions about $\ulambda$-spaces. The above theorem shows that the linear version of the Ryll-Nardzewski theorem does indeed hold for length two biquadratic spaces.

\subsection{Questions}

We mention a few questions arising from this work:
\begin{itemize}
\item In our proofs, it is important the coefficient field is uncountable due to our use of Schur's lemma (see \S \ref{ss:schur}). Can this assumption be removed?
\item Do the conclusions of Theorem~\ref{mainthm2} apply to all finite length multi-quadratic spaces? This seems plausible to us.
\item Can one classify irreducible $\ulambda$-spaces for general $\ulambda$?
\end{itemize}
We offer one observation that suggests that the multi-quadratic case might be quite special. Let $V$ be the cubic space with form $\sum_{i \ge 1} x_i^3$. The automorphism group of $V$ is the wreath product of the infinite symmetric group and the group of third roots of unity. This group is not linearly oligomorphic, but does act irreducibly on $V$. Thus, for general $\ulambda$, the class of finite length $\ulambda$-spaces will be strictly larger than the class of linearly oligomorphic $\ulambda$-spaces.

\subsection{Outline}

In \S \ref{s:prelim}, we review background on tensor spaces. In \S \ref{s:multiq}, we classify the irreducible multi-quadratic spaces. In \S \ref{s:type1}--\ref{s:III} we analyze the seven families of bi-quadratic spaces appearing in our main theorem. Finally, the main theorem is proved in \S \ref{s:mainthm}.

\subsection{Notation}

All vector spaces have dimension at most $\aleph_0$, unless otherwise mentioned.

\begin{description}[align=right,labelwidth=2.25cm,leftmargin=!]
\item[ $k$ ] the coefficient field (see \S \ref{s:multiq})
\item[ $\bP^1$ ] the set $k \cup \{\infty\}$
\item[ $\ker(\omega)$ ] the nullspace of a quadratic form $\omega$
\item[ $\ker(V)$ ] $\ker(\omega) \cap \ker(\omega')$ if $V$ is biquadratic
\item[ $\rM_2(k)$ ] the algebra of $2 \times 2$ matrices over $k$
\end{description}

\subsection*{Acknowledgments}

We thank Arthur Bik and Jan Draisma for helpful conversations.

\section{Preliminaries on tensor spaces} \label{s:prelim}

\subsection{Fra\"iss\'e theory} \label{ss:fraisse}

Let $\cC$ be a category in which all morphisms are monomorphisms; we will refer to morphisms as ``embeddings.'' An \defn{ind-object}\footnote{Of course, one can consider ind-objects indexed by sets other than $\bN$, but we will not need them.} of $\cC$ is a diagram $X_1 \to X_2 \to \cdots$ in $\cC$. The ind-objects in $\cC$ naturally form a category, which contains $\cC$ as a full subcategory; see \cite[\S A.2]{homoten} for details.

We now introduce three important classes of objects. We say that an ind-object $X$ is \defn{universal} if every object of $\cC$ embeds into $X$. We say that $X$ is \defn{homogeneous} if whenever $Y$ is an object of $\cC$ and $\alpha, \beta \colon Y \to X$ are two embeddings, there exists an automorphism $\sigma$ of $X$ such that $\beta=\sigma \circ \alpha$. We say that $X$ is \defn{f-injective} if, given embeddings $\alpha \colon Z \to Y$ and $\gamma \colon Z \to X$, with $Y$ and $Z$ objects of $\cC$, there exists an embedding $\beta \colon Y \to X$ such that $\gamma=\beta \circ \alpha$. We note that if $\cC$ has an initial object then any f-injective object is universal (take $Z$ to be the initial object). The following basic result relates these notions:

\begin{proposition} \label{prop:fraisse}
A universal ind-object is f-injective if and only if it is homogeneous. Moreover, any two universal homogeneous ind-objects are isomorphic.
\end{proposition}

\begin{proof}
See \cite[Proposition~A.7]{homoten}.
\end{proof}

Categorical Fra\"iss\'e theory provides an important tool (the Fra\"iss\'e limit) for constructing a universal homogeneous object, under suitable hypotheses; see \cite[\S A.6]{homoten} for details. We will not need this, but we will use the concepts of f-injective and homogeneous objects.

\begin{example} \label{ex:Q}
Let $\cC$ be the category of finite totally ordered sets, where embeddings are monotonic functions. An ind-object of $\cC$ is (equivalent to) a countable totally ordered set. It is not difficult to see that a countable dense totally ordered set is universal and f-injective. Thus Proposition~\ref{prop:fraisse} implies that any two such sets are isomorphic, which is a classical theorem of Cantor.
\end{example}

\begin{example}
Let $\cC$ be the category of finite dimensional vector spaces equipped with a symmetric bilinear form, over a fixed field of characteristic $\ne 2$; embeddings are injective linear maps that respect the form. An ind-object of $\cC$ is (equivalent to) a vector space of countable dimension equipped with a form. Witt's theorem states that if $V$ is a non-degenerate object of $\cC$ (i.e., the form has zero nullspace) then $V$ is homogeneous. This continues to hold for ind-objects as well, and so the universal homogeneous ind-object is any space with a form that has zero null-space.
\end{example}

\subsection{Tensor spaces}

Let $k$ be an infinite field and let $V$ be a $k$-vector space. Given a partition $\lambda$ such that $\vert \lambda \vert !$ is invertible in $k$, a \defn{$\lambda$-form} on $V$ is a linear map $\bS_{\lambda}(V) \to k$, where $\bS_{\lambda}$ denotes the Schur functor. Some examples:
\begin{itemize}
\item If $\lambda$ is the empty partition then $\bS_{\lambda}(V)=k$, and so a $\lambda$-form is simply a scalar. This case is obviously somewhat degenerate.
\item If $\lambda=(1)$ then a $\lambda$-form is a linear functional on $V$.
\item If $\lambda=(2)$ then a $\lambda$-form can be viewed as either a quadratic form or a symmetric bilinear form. We will freely pass between the two interpretation.
\item More generaly, if $\lambda=(d)$ is a one row partition then a $\lambda$-form can be viewed as either a degree $d$ homogeneous polynomial, or as a symmetric $d$-linear form on $V$.
\end{itemize}
In this paper, we will only need the cases $(1)$ and $(2)$, though we state some background results in greater generality.

Given a tuple of partitions $\ulambda=[\lambda_1, \ldots, \lambda_n]$, such that each $\vert \lambda_i \vert!$ is invertible, a \defn{$\ulambda$-space} is a vector space $V$ equipped with a $\lambda_i$-form for each $1 \le i \le n$. An \defn{embedding} of $\ulambda$-spaces is an injective linear map that respects all the forms. All $\ulambda$-spaces we consider throughout this paper have dimension $\le \aleph_0$. We let $\cC_{\ulambda}$ be the category of $\ulambda$-spaces of dimension $\le \aleph_0$, and we let $\cC^{\rf}_{\ulambda}$ be the subcategory of finite dimensional spaces. We note that $\cC_{\ulambda}$ is equivalent to the category of ind-objects in $\cC^{\rf}_{\ulambda}$. We can thus apply the terminology from \S \ref{ss:fraisse} to $\ulambda$-spaces.

Let $V$ be a $k$-vector space, and let $f \colon V \to k$ be a $(d)$-form, thought of as a polynomial. We say that $f$ has \defn{strength $\le s$} if there is an expression
\begin{displaymath}
f = \sum_{i=1}^s g_i h_i
\end{displaymath}
where $g_i$ and $h_i$ are homogeneous polynomial functions on $V$ with degree $<d$, i.e., $g_i$ is an $(e_i)$-form and $h_i$ is a $(d_i-e_i)$-form, with $0<e_i<d_i$. Given a collection $f_1, \ldots, f_r$ of $(d)$-forms, we define the \defn{collective strength} to be the minimal strength of a non-trivial linear combination. Given a collection $f_1, \ldots, f_r$ of forms of varying degree, we define the \defn{collective strength} to be the minimum collective strength in each degree separately.

\begin{example}
We give a few examples of strength. A non-zero linear form has infinite strength. More generally, a collection of linear forms has infinite collective strength if it is linearly dependent, and otherwise has collective strength zero. If $f$ is a $(2)$-form then the strength of $f$ is equal to $\lceil \tfrac{1}{2} \operatorname{rank}(f) \rceil$.
\end{example}

We now recall an important theorem of Kazhdan and Ziegler \cite{KaZ} on universal $\ulambda$-spaces.

\begin{theorem} \label{thm:univ}
Suppose $\ulambda=[(d_1), \ldots, (d_r)]$, where each $d_i$ is a positive integer. Assume $k$ is algebraically closed and $d_i!$ is invertible in $k$ for each $i$. Let $(V, f_1, \ldots, f_r)$ be a $\ulambda$-space. 
\begin{enumerate}
\item Given $n \in \bN$ there exists $s=s(k, \ulambda, n)$ with the following property: if $(f_1, \ldots, f_r)$ has collective strength at least $s$ then any $n$-dimensional $\ulambda$-space embeds into $V$.
\item If $(f_1, \ldots, f_r)$ has infinite collective strength then $V$ is universal.
\end{enumerate}
\end{theorem}

Kazhdan and Ziegler only work in the finite dimensional setting, and thus only state version (a). The infinite dimensional version, as well as a generalization to arbitrary tuples $\ulambda$, is given in \cite{BDDE}. We will mostly use the following corollary of this result:

\begin{corollary} \label{cor:univ}
In the setting of the theorem, suppose that the $f_i$'s have collective strength at least $s(k, \ulambda, 1)$. Given $c_1, \ldots, c_r \in k$, there exists $v \in V$ such that $f_i(v)=c_i$ for each $1 \le i \le n$.
\end{corollary}

\begin{proof}
Let $(L, f'_1, \ldots, f'_r)$ be a 1-dimensional $\ulambda$-space with basis $w$, and forms defined by $f'_i(w)=c_i$. By the theorem, there is an embedding of $\ulambda$-spaces $L \to V$. The image of $w$ is the requisite vector $v$.
\end{proof}

\begin{remark}
One can also prove the corollary using commutative algebraic properties of infinite strength polynomials proved in \cite{AH,ESS}; see \cite[Theorem~2.1]{isocubic}.
\end{remark}

The following is the second theorem on $\ulambda$-spaces that will be important to us. It was proven in \cite{homoten} using the Fra\"iss\'e limit.

\begin{theorem} \label{thm:homo}
Let $\ulambda=[\lambda_1, \ldots, \lambda_n]$ be a tuple of non-empty partitions such that each $\vert \lambda_i \vert!$ is invertible in $k$. Then there exists a universal homogeneous $\ulambda$-space of countable dimension, and it is unique up to isomorphism.
\end{theorem}

We will give an explicit description of the universal homogeneous space in the multi-quadratic case below (\S \ref{ss:mqchar}).

\subsection{Linearly oligomorphic groups} \label{ss:linolig}

An action of a group $G$ on a set $X$ is \defn{oligomorphic} if $G$ has finitely many orbits on $X^n$ for all $n \ge 0$. This condition is closely related to homogeneity, as many homogeneous relational structures have oligomorphic automorphism group. For instance, if $X$ is a countable dense totally ordered set (Example~\ref{ex:Q}) then $\Aut(X)$ acts oligomorphically on $X$. See \cite{Cameron} for general background.

In \cite{homoten}, we formulated a linear analog of the oligomorphic condition, which we now recall. Suppose that $V$ is a linear representation of a group $G$. We say that $V$ is \defn{linearly oligomorphic} if for every $d \ge 0$ there exists a finite dimensional subspace $E=E(d)$ of $V$ with the following property: if $W$ is a $d$-dimensional subspace of $V$ then there exists $g \in G$ such that $gW \subset E$. This condition can be rephrased to be a bit more succinct in the following way. Let $\Gr_d(V)$ denote the set of $d$-dimensional subspaces of $V$. Then $V$ is linearly oligomorphic if for every $d \ge 0$ there exists a finite dimensional $E \subset V$ such that the natural map
\begin{displaymath}
\Gr_d(E) \to \Gr_d(V)/G
\end{displaymath}
is surjective. The set $\Gr_d(E)$ is (the $k$-points of) a finite dimensional variety, and so the linearly oligomorphic condition intuitively means that $\Gr_d(V)/G$ is a finite dimensional space for all $d$. A \defn{linearly oligomorphic group} is a group equipped with a faithful linearly oligomorphic representation.

The following theorem from \cite{homoten} points to the importance of these groups:

\begin{theorem} \label{thm:oligo}
Let $V$ be the universal homogeneous $\ulambda$-space, in the context of Theorem~\ref{thm:homo}. Then $\Aut(V)$ is linearly oligomorphic (with respect to its action on $V$).
\end{theorem}

We now establish a few more simple properties of linearly oligomorphic groups.

\begin{proposition}
A linearly oligomorphic representation $V$ of a group $G$ has finite length.
\end{proposition}

\begin{proof}
Let $E$ be a finite dimensional subspace of $V$ such that $\Gr_1(E) \to \Gr_1(V)/G$ is surjective. We claim that the map
\begin{displaymath}
\{ \text{$G$-subrepresentations of $V$} \} \to \{ \text{subspaces of $E$} \}, \qquad
W \mapsto W \cap E
\end{displaymath}
is injective; this will show that $V$ has length $\le \dim(E)$. Thus suppose that $W_1$ and $W_2$ are two $G$-subrepresentations such that $W_1 \cap E = W_2 \cap E$. Given $x \in W_1$, there exists $g \in G$ such that $gx \in E$. Since $W_1$ is a subrepresentation, we have $gx \in W_1$. Thus $gx$ belongs to $W_1 \cap E = W_2 \cap E$, and so $gx \in W_2$, and so $x \in W_2$. Thus shows $W_1 \subset W_2$, and the reverse containment is similar.
\end{proof}

\begin{proposition} \label{prop:uh-irred}
Suppose we are in the setting of Theorem~\ref{thm:oligo}, and $\vert \lambda_i \vert \ge 2$ for all $i$. Then $V$ is an irreducible representation of $G=\Aut(V)$.
\end{proposition}

\begin{proof}
Let $\omega_1, \ldots, \omega_r$ be the given forms on $V$. Suppose $W$ is a proper non-zero subspace of $V$. Let $\eta \colon V \to k$ be a non-zero linear functional such that $W$ is contained in $\ker(\eta)$. Let $x$ be a non-zero element of $W$. Since $\omega_1, \ldots, \omega_r, \eta$ has infinite strength, by Corollary~\ref{cor:univ}, we can find $y \in V$ such that $\omega_i(y)=\omega_i(x)$ for all $1 \le i \le r$ and $\eta(y)=1$. Observe that $y \not\in W$, since $\eta(y)=1$. Since $V$ is homogeneous, there exists $g \in G$ such that $gx=y$. This shows that $W$ is not stable by $G$, which completes the proof.
\end{proof}

\begin{proposition} \label{prop:sum-oligo}
Suppose that $V$ is a linearly oligomorphic representation of $G$. Then $V^{\oplus n}$ is also a linearly oligomorphic representation of $G$, for any $n \ge 0$.
\end{proposition}

\begin{proof}
Let $d \ge 0$ be given. Let $E$ be a finite dimensional subspace of $V$ such that $\Gr_m(E)$ surjects onto $\Gr_m(V)/G$ for all $0 \le m \le nd$; this exists since the action of $G$ on $V$ is linearly oligomorphic. Let $W$ be a $d$-dimensional subspace of $V^{\oplus n}$, let $W_i \subset V$ be the $i$th projection of $W$, and let $W'=\sum_{i=1}^n W_i$. Then $\dim(W') \le nd$, and so there exists $g \in G$ such that $g W' \subset E$. It follows that $gW_i \subset E$ for each $i$, and so $gW \subset E^{\oplus n}$. We thus see that $\Gr_d(E^{\oplus n})$ surjects onto $\Gr_d(V^{\oplus n})/G$, which completes the proof.
\end{proof}

We do not know if a direct sum of two linearly oligomorphic representations of the same group is linearly oligomorphic, in general. However, we do have the following result:

\begin{proposition} \label{prop:sum-oligo-2}
Suppose $V$ and $W$ are linearly oligomorphic representations of groups $G$ and $H$. Then $V \oplus W$ is a linearly oligomorphic representation of $G \times H$.
\end{proposition}

\begin{proof}
Let $d$ be given. Let $E$ and $F$ be $d$-dimensional subspaces of $V$ and $W$ such that $\Gr_d(E) \to \Gr_d(V)/G$ and $\Gr_d(F) \to \Gr_d(W)/H$ are surjective. Let $U$ be a $d$-dimensional subspace of $V \oplus W$. We can then find $g \in G$ and $h \in H$ such that $g\pi_1(U) \subset E$ and $h\pi_2(U) \subset F$, where $\pi_i$ is the projection map. Thus $(g,h) U \subset E \oplus F$, and so the map
\begin{displaymath}
\Gr_d(E \oplus F) \to \Gr_d(V \oplus W)/(G \times H)
\end{displaymath}
is surjective, as required.
\end{proof}

\subsection{First-order theories} \label{ss:logic}

Fix a tuple\footnote{We work with one row partitions to keep the discussion simple; see \cite[\S 5.1]{homoten} for the general case.} $\ulambda=[(d_1), \ldots, (d_r)]$ where each $d_i$ is a positive integer such that $d_i!$ is invertible in $k$. We will regard $(d)$-forms as symmetric multilinear forms in what follows. We now recall how to discuss $\ulambda$-spaces from the point of view of model theory. We refer to \cite[\S 5.1]{homoten} for a more detailed discussion; however, the examples discussed below should also help clairfy the key ideas.

We consider the following first-order language. There are two sorts of variables, namely, scalars ($\alpha$, $\beta$, and so on) and vectors ($x$, $y$, and so on). There are binary function symbols for addition and multiplication of scalars, as well as constant scalar symbols for each element of $k$. There are also binary function symbols for addition of vectors, and scalar multiplication, and a constant vector symbol 0. Finally, there are $r$ function symbols $\omega_1, \ldots, \omega_r$. The symbol $\omega_i$ takes $d_i$ vector inputs and produces one scalar output. It is also multi-linear and symmetric.

Recall that a \defn{formula} is a grammatical expression in our first-order language that is allowed to have free variables, while a \defn{sentence} is a formula with no free variables. Suppose that $V$ is a $\ulambda$-space. If $\phi(x_1, \ldots, x_n)$ is a formula, then it defines a subset of $V^n$, namely, the set of tuples where $\phi$ is true. We call such subsets \defn{definable}. If $\phi$ is a sentence then it is either true or false on $V$. The collection of all true sentences is called the \defn{theory} of $V$, and denoted $\Th(V)$.

\begin{example}
Consider the formula
\begin{displaymath}
\theta_2(x_1,x_2) \colon \forall \alpha_1, \alpha_2 (\alpha_1 x_1 + \alpha_2 x_2 \implies \alpha_1=0 \land \alpha_2=0).
\end{displaymath}
In this formula, $x_1$ and $x_2$ are free variables, since they are not quantified over. The formula $\theta_2$ expresses the linear independence of $x_1$ and $x_2$; that is, if $V$ is a $\ulambda$-space then
\begin{displaymath}
\{ (x_1,x_2) \in V^2 \mid \theta_2(x_1,x_2) \}
\end{displaymath}
is the set of linearly independent pairs of vectors. Now consider the formula
\begin{displaymath}
\exists x_1, x_2 (\theta_2(x_1,x_2)).
\end{displaymath}
This is a sentence since there are no free variables. This sentence belongs to $\Th(V)$ if and only if $\dim(V) \ge 2$. There is a similar formula $\theta_n(x_1, \ldots, x_n)$ that detects linear independence of $n$ vectors. We thus see that from $\Th(V)$ one can determine if $V$ is finite dimensional, and if so, what its dimension is.
\end{example}

\begin{example}
Suppose $\ulambda=[(2)]$. Consider the formula
\begin{displaymath}
\psi(x) \colon \forall y (\omega(x,y)=0).
\end{displaymath}
This formula expresses that $x$ belongs to $\ker(\omega)$, the nullspace of $\omega$. In particular, we see that if $V$ is a $\ulambda$-space then $\ker(\omega)$ is a definable subspace of $V$. Combined with the previous example, we see that from $\Th(V)$ one can determine if $\ker(\omega)$ is finite dimensional, and if so, what its dimension is. Indeed, $\dim(\ker(\omega)) \ge n$ if and only if $\Th(V)$ contains the sentence
\begin{displaymath}
\exists x_1, \ldots, x_n (\psi(x_1) \land \cdots \land \psi(x_n) \land \theta_n(x_1, \ldots, x_n)). \qedhere
\end{displaymath}
\end{example}

Suppose $V$ is a $\ulambda$-space; as always, we assume $\dim(V) \le \aleph_0$. We say that $V$ is \defn{linearly $\omega$-categorical}\footnote{The $\omega$ here refers to the first infinite ordinal, and not a multilinear form.} if whenever $W$ is a $\ulambda$-space with $\Th(V)=\Th(W)$ (and $\dim(W) \le \aleph_0$), we have an isomorphism $V \cong W$. Essentially, this means that the isomorphism type of $V$ admits a (possibly infinite) first-order axiomatization. For example, if $k$ is algebraically closed and $\ulambda=[(2)]$ then a $\ulambda$-space is determined up to isomorphism by its rank and the dimension of its nullspace; since these one can recover these from the theory, any $\ulambda$-space is linearly $\omega$-categorical.

The Ryll-Nardzewski theorem is a classical result in model theory, which states that if $X$ is a countable relational structure then $X$ is $\omega$-categorical if and only if $\Aut(X)$ acts oligomorphically on $X$. The statement admits a natural analog for tensor spaces:

\begin{conjecture} \label{conj:RN}
A $\ulambda$-space $V$ is linearly $\omega$-categorical if and only if $\Aut(V)$ is linearly oligomorphic.
\end{conjecture}

In \cite[Theorem~5.13]{homoten}, we show that the universal homogeneous $\ulambda$-space is linearly $\omega$-categorical. Combined with Theorem~\ref{thm:oligo}, this verifies Conjecture~\ref{conj:RN} for these spaces. The present paper adds a bit more evidence for the conjecture.

\section{Irreducible multi-quadratic spaces} \label{s:multiq}

\textit{For the remainder of the paper, $k$ denotes an uncountable algebraically closed field of characteristic $\ne 2$.}

\subsection{Characterization of homogeneity} \label{ss:mqchar}

For a quadratic form $\omega$ on a vector space $V$ and an element $x \in V$, we let $\omega_x$ be the linear functional $\omega(x, -)$. Let $(V, \omega_1, \ldots, \omega_r)$ be a multi-quadratic space. Define $\Delta=\Delta(\omega_1, \ldots, \omega_r)$ to be the set of tuples $(x_1, \ldots, x_r)$ such that $\sum_{i=1}^r (\omega_i)_{x_i} = 0$. This is a linear subspace of $V^r$.

\begin{theorem} \label{thm:mqchar}
The space $V$ is universal homogeneous if and only if $V$ is infinite dimensional and $\Delta=0$.
\end{theorem}

\begin{proof}
If $V$ is finite dimensional it is clearly not universal. Suppose $\Delta \ne 0$. Let $(x_1, \ldots, x_r)$ be a non-zero element of $\Delta$, and let $W$ be the subspace of $V$ spanned by the $x_i$'s. Without loss of generality, we suppose that $x_1, \ldots, x_s$ form a basis of $W$. Define a new multi-quadratic space $(W', \omega'_1, \ldots, \omega'_r)$ as follows. The underlying vector space is $W'=W \oplus ky$, where $y$ is a new basis vector. The form $\omega'_i$ restricts to $\omega_i$ on $W$. For $1 \le i \le r$ and $1 \le j \le s$, we define $\omega'_i(x_j, y)$ to be~1 if $i=j=1$, and~0 otherwise; note that $\sum_{i=1}^r \omega'_i(x_i, y) = 1$. We also put $\omega'_i(y,y)=0$ for all $i$. The embedding $W \to V$ clearly cannot extend to $W'$, and so $V$ is not f-injective.

Now suppose that $\Delta=0$ and $V$ is infinite dimensional. We first claim that $(\omega_1, \ldots, \omega_r)$ has infinite collective strength. Indeed, suppose a non-trivial linear combination $\omega=\sum_{i=1}^r \alpha_i \omega_i$ has finite rank. Let $x \in V$ be a non-zero element of the nullspace of $\omega$, which exists since $V$ is infinite dimensional. Then $(\alpha_1 x, \ldots, \alpha_r x)$ is a non-zero element of $\Delta$, a contradiction. This proves the claim.

We now show that $V$ is f-injective. Let $W \subset V$ be a finite dimensional subspace, and let $W \to W'$ be an embedding where $W$ has codimension one in $W'$. We must show that we can compatibly embed $W'$ into $V$. Let $x_1, \ldots, x_n$ be a basis for $W$, and let $y$ be a vector in $W'$ that does not belong to $W$. Note that the $rn$ linear functionals $\omega_i(x_j, -)$ on $V$ are linearly independent since $\Delta=0$. Let $\lambda$ be a linear functional on $V$ that vanishes on $W$ and is linearly independent from these functionals; this exists since the space of functionals vanishing on $W$ is infinite dimensional. Applying Corollary~\ref{cor:univ}, we can find $z \in V$ such that
\begin{displaymath}
\omega_i(z,z) = \omega_i(y,y), \qquad
\omega_i(x_j,z) = \omega_i(x_j,y), \qquad
\lambda(z)=1.
\end{displaymath}
for all $1 \le i \le r$ and $1 \le j \le n$. Indeed, the previous paragraph shows that the degree two equations here have infinite collective strength, while we have just seen that the linear equations are linearly independent. Since $\lambda(z)=1$, it follows that $z$ does not belong to $W$. We can now embed $W'$ into $V$ by $y \mapsto z$, which completes the proof.
\end{proof}

\begin{remark}
The theorem can be restated in the following more concrete manner. Let $V$ be a vector space of countable infinite dimension equipped with quadratic forms $\omega_1, \ldots, \omega_r$. Choose a basis for $V$ and let $A_i$ be the matrix for the form $\omega_i$. Then $(V, \omega_1, \ldots, \omega_r)$ is universal homogeneous if and only if the columns of the $A_i$'s (taken together) are linearly independent. Indeed, this is simply a translation of the $\Delta=0$ condition.
\end{remark}

\begin{remark}
Any two universal homogeneous $r$-quadratic spaces are isomorphic (Proposition~\ref{prop:fraisse}). In light of the previous remark, this can be cast into concrete terms as follows. Let $(A_1, \ldots, A_r)$ be symmetric matrices with rows and columns indexed by $\{1,2,3,\ldots\}$ such that the columns of the $A_i$'s (taken together) are linearly independent. Let $(B_1, \ldots, B_r)$ be a second such tuple. Then there is an invertible matrix $g$ such that $B_i=gA_ig^t$ for each $1 \le i \le r$.
\end{remark}

\subsection{Schur's lemma} \label{ss:schur}

We now recall a version of Schur's lemma in the infinite dimensional setting. For the next set of results, $G$ is a group. The hypothesis we have imposed on $k$ (uncountable and algebraically closed) is crucial in these results.

\begin{proposition} \label{prop:schur}
Let $V$ be an irreducible representation of $G$ over $k$ of countable dimension. Then $\End_G(V)=k$.
\end{proposition}

\begin{proof}
Let $T \colon V \to V$ be an operator commuting with $G$. Since $\End_G(V)$ is a division algebra, it contains the field $k(T)$, and so $V$ has the structure of a $k(T)$-vector space. Since $V$ has countable dimension over $k$, it follows that $k(T)$ must also have countable dimension over $k$. Since any proper extension of $k$ has uncountable dimension over $k$, it follows that $k(T)=k$, and so $T$ is a scalar.
\end{proof}

\begin{corollary} \label{cor:schur1}
Let $V$ and $W$ be finite length representations of $G$ over $k$ of countable dimensions. Then $\Hom_G(V,W)$ is finite dimensional as a $k$-vector space.
\end{corollary}

\begin{proof}
By d\'evissage, we reduce to the case where $V$ and $W$ are irreducible. If $V \cong W$ then $\Hom_G(V,W) \cong \End_G(V)$ is one-dimensional. Otherwise, $\Hom_G(V,W)=0$.
\end{proof}

\begin{corollary} \label{cor:schur2}
Let $V$ be an irreducible representation of $G$ over $k$ of countable dimension. Then the subrepresentations of $V^{\oplus n} = V \otimes k^n$ are exactly those of the form $V \otimes E$, where $E$ is a subspace of $k^n$.
\end{corollary}

\begin{proof}
Let $W$ be a subrepresentation of $V^{\oplus n}$. Since $V^{\oplus n}$ is semi-simple and isotypic of type $V$, the same is true of $W$. It follows that $W$ is the sum of images of maps $V \to V^{\oplus n}$, and any such map has the form $x \mapsto (\alpha_1 x, \ldots \alpha_n x)$ by the proposition. The result thus follows.
\end{proof}

\subsection{The classification}

We are now ready to classify irreducible multi-quadratic spaces.

\begin{theorem} \label{thm:mqclass}
Let $(V, \omega_1, \ldots, \omega_r)$ be an $r$-quadratic space of infinite dimension, let $\Omega$ be the span of the $\omega_i$'s, and let $\eta_1, \ldots, \eta_s$ be a basis for $\Omega$. Then $V$ is irreducible if and only if $(V, \eta_1, \ldots, \eta_s)$ is a universal homogeneous $s$-quadratic space.
\end{theorem}

\begin{proof}
Let $G$ be the automorphism group of $(V, \omega_{\bullet})$; note that this is also the automorphism group of $(V, \eta_{\bullet})$. If the latter is universal homogeneous then we have already seen that $V$ is an irreducible representation of $G$ (Proposition~\ref{prop:uh-irred}).

Now suppose that $V$ is an irreducible representation of $G$. Let $\Delta=\Delta(\eta_1, \ldots, \eta_s)$. We claim $\Delta=0$, which will show that $(V, \eta_{\bullet})$ is universal homogeneous (Theorem~\ref{thm:mqchar}). Suppose $\Delta \ne 0$. Since $\Delta$ is a $G$-subrepresentation of $V^s$ and $V$ is irreducible, it follows that $\Delta$ contains a non-zero element of the form $(\alpha_1 x, \ldots, \alpha_s x)$, where $\alpha_1, \ldots, \alpha_s \in k$ and $x \in V$ (Corollary~\ref{cor:schur2}). Let $\eta'=\sum_{i=1}^s \alpha_i \eta_i$. We have
\begin{displaymath}
0 = \sum_{i=1}^s (\eta_i)_{\alpha_i x} = \eta'_x,
\end{displaymath}
i.e., $x$ belongs to $\ker(\eta')$. We thus see that $\ker(\eta')$ is a non-zero $G$-subrepresentation of $V$, and therefore all of $V$ by irreducibility. But this means $\eta'=0$, contradicting the linear independence of the $\eta_i$'s. This completes the proof.
\end{proof}

We now deduce some consequences of the theorem. Define $\Theta(V)$ or $\Theta(\omega_1, \ldots, \omega_r)$ to be the set of points $\alpha \in k^r$ such that $\sum_{i=1}^r \alpha_i \omega_i=0$.

\begin{corollary} \label{cor:mqclass-1}
Let $V$ and $V'$ be infinite dimensional irreducible $r$-quadratic spaces. Then $V$ and $V'$ are isomorphic if and only if $\Theta(V)=\Theta(V')$.
\end{corollary}

\begin{proof}
Clearly, if $V$ and $V'$ are isomorphic then $\Theta(V)=\Theta(V')$. Now suppose $\Theta(V)=\Theta(V')$. Let $\omega_1, \ldots, \omega_r$ be the forms on $V$, let $\Omega$ be their span, let $\eta_1, \ldots, \eta_s$ be a basis of $\Omega$, and write $\omega_i=\sum_{j=1}^s a_{i,j} \eta_j$. Similarly define $\omega'_{\bullet}$, $\Omega'$, $\eta'_{\bullet}$, and $a'_{\bullet,\bullet}$. Since $\Theta(V)=\Theta(V')$, we have $\dim(\Omega)=\dim(\Omega')$; moreover, we can choose the $\eta$ and $\eta'$ such that $a_{i,j}=a'_{i,j}$. Since $(V, \eta_{\bullet})$ and $(V', \eta'_{\bullet})$ are both universal homogeneous, there is an isomorphism $\phi \colon V \to V'$ such that $\phi^*(\eta'_i)=\eta_i$. Clearly then $\phi^*(\omega'_i)=\omega_i$, an so $\phi$ is an isomorphism of $r$-quadratic spaces.
\end{proof}

\begin{corollary}
An infinite dimensional irreducible multi-quadratic space $V$ is linearly $\omega$-categorical.
\end{corollary}

\begin{proof}
Suppose $V'$ is a multi-quadratic space with $\Th(V)=\Th(V')$. Then $V'$ is infinite dimensional and $\Theta(V')=\Theta(V)$. Using notation as in the proof of the previous corollary, we again have $\dim(\Omega)=\dim(\Omega')$, and can choose $\eta_{\bullet}$ and $\eta'_{\bullet}$ such that $a=a'$. We have $\Delta(\eta'_{\bullet})=\Delta(\eta_{\bullet})$ since this can be read off from the theory. The space $\Delta(\eta_{\bullet})$ vanishes by Theorems~\ref{thm:mqclass} and~\ref{thm:mqchar}, and so $\Delta(\eta'_{\bullet})$ vanishes as well; Theorem~\ref{thm:mqchar} thus shows that $(V', \eta'_{\bullet})$ is universal homogeneous, and thus irreducible. Hence $(V', \omega'_{\bullet})$ is irreducible too. Since $\Theta(V)=\Theta(V')$, we have $V \cong V'$ by the previous corollary.
\end{proof}

\begin{corollary} \label{cor:mqclass-1}
If $V$ is an irreducible multi-quadratic space then $\Aut(V)$ is linearly oligomorphic.
\end{corollary}

\begin{proof}
Theorem~\ref{thm:mqclass} shows that the automorphism group of $V$ coincides with the automorphism group of a universal homogeneous multi-quadratic structure on $V$, which is known to be linearly oligomorphic by Theorem~\ref{thm:oligo}.
\end{proof}

\section{Type I spaces} \label{s:type1}

\subsection{Set-up}

In \S \ref{s:type1}--\ref{s:III}, we analyze the seven families of biquadratic spaces appearing in the main theorem. In each case, we give a first-order definition of the spaces, which we prove uniquely characterizes them, and analyze the automorphism group. The most interesting cases are certainly Type~IIb (\S \ref{ss:IIb}) and Type~III (\S \ref{s:III}), though it is necessary to treat all cases.

We now introduce some notation and terminology that we use throughout this discussion:
\begin{itemize}
\item We write $\ker(\omega)$ for the nullspace of a quadratic form $\omega$.
\item As defined above, for a quadratic form $\omega$ on $V$ and $x \in V$ we write $\omega_x$ for $\omega(x, -)$.
\end{itemize}
In what follows, $(V, \omega, \omega')$ is a biquadratic space.
\begin{itemize}
\item We put $\ker(V)=\ker(\omega) \cap \ker(\omega')$.
\item We define the \defn{total orthogonal complement} of a subspace $W$ of $V$ to be the set of all vectors orthogonal to $W$ under both $\omega$ and $\omega'$. We note that the total orthogonal complement of $V$ itself is $\ker(V)$.
\item For a point $p=[p_0:p_1]$ in $\bP^1$, we let $\omega(p)=p_1 \omega+p_0 \omega'$. The form $\omega(p)$ is only defined up to scalar multiples, but, for example, $\ker(\omega(p))$ is well-defined. We note that $\omega(0)=\omega$ and $\omega(\infty)=\omega'$ (up to scalars).
\item Let
\begin{displaymath}
g = \begin{pmatrix} a & b \\ c & d \end{pmatrix}
\end{displaymath}
be an element of $\GL(2)$. We define $V^g$ to be the biquadratic space $(V, \eta, \eta')$ where
\begin{displaymath}
\eta = a \omega + b \omega', \qquad \eta' = c \omega + d \omega'.
\end{displaymath}
We refer to $V^g$ as the \defn{twist} of $V$ by $g$.
\item If $W$ is a second biquadratic space, we define $V \oplus W$ to be the biquadratic space equipped with the forms extending those form $V$ and $W$, and for which $V$ and $W$ are orthogonal.
\end{itemize}

\subsection{Type Ia} \label{ss:Ia}

The following definition introduces this class of spaces.

\begin{definition}
A biquadratic space $V$ is \defn{Type Ia} if there are points $p \ne q$ in $\bP^1$ such that the nullspaces $W$ and $W'$ of $\omega(p)$ and $\omega(q)$ are both infinite dimensional and $V=W \oplus W'$. The \defn{parameter} of $V$ is the 2-element subset $\{p,q\}$ of $\bP^1$.
\end{definition}

We show below that the parameter is well-defined. This is a first-order definition. By this, we mean that if $V$ is a biquadratic space, one can determine from $\Th(V)$ if $V$ satisfies the definition; one can also read off the parameter of $V$ from $\Th(V)$.

%
%

\begin{construction} \label{con:Ia}
Let $V_1$ be an infinite dimensional biquadratic space where $\omega$ is non-degenerate and $\omega'=0$, and let $V_2$ be defined similarly but with the roles of $\omega$ and $\omega'$ reversed. Then $V=V_1 \oplus V_2$ is a Type~Ia space with parameters $\{0, \infty\}$. Twisting by $\GL(2)$ gives examples realizing any parameters.
\end{construction}

\begin{proposition}
Two Type~Ia spaces are isomorphic if and only if their parameters coincide. In particular, the parameter of a Type~Ia space is well-defined.
\end{proposition}

\begin{proof}
Let $V$ be a Type~Ia space; twisting by $\GL(2)$, we assume the parameter is $\{0,\infty\}$. We claim that $V$ is isomorphic to the space from Construction~\ref{con:Ia}. Let $W$ and $W'$ be the nullspaces of $\omega$ and $\omega'$. By assumption, both $W$ and $W'$ are infinite dimensional, and $V=W \oplus W'$. Note that this is an orthogonal decomposition with respect to both $\omega$ and $\omega'$. We have $V_1 \cong W'$ and $V_2 \cong W$ as bi-quadratic spaces, where the $V_i$'s are as in Construction~\ref{con:Ia}, and so the claim follows. It is now clear that $\omega(p)$ has non-zero nullspace only if $p=0$ or $p=\infty$, which shows that the parameter of $V$ is well-defined. Of course, isomorphic spaces have the same parameter.
\end{proof}

\begin{corollary}
A Type~Ia space is linearly $\omega$-categorical.
\end{corollary}

\begin{proof}
Let $V$ be a Type~Ia space and let $V'$ be a biquadratic space such that $\Th(V)=\Th(V')$. Since the Type~Ia condition and parameter can be read off from the theory, it follows that $V'$ is Type~Ia with the same parameter as $V$. Thus $V'$ is isomorphic to $V$ by the proposition, which shows that $V$ is linearly $\omega$-categorical.
\end{proof}

The exact argument used in the above corollary applies for the other types of spaces considered below, so we will not repeat it.

\begin{proposition}
Let $V$ be a Type~Ia space, and let $G=\Aut(V)$. Then $G$ is linearly oligomorphic, $V$ is a length two representation of $G$, and $\End_G(V)=k \oplus k$.
\end{proposition}

\begin{proof}
It suffices to treat the space $V=V_1 \oplus V_2$ from Construction~\ref{con:Ia}. It is clear that $G$ contains $\Aut(V_1) \times \Aut(V_2)$, and each $\Aut(V_i)$ is a copy of the infinite orthogonal group. Since $V_1$ is the nullspace of $\omega'$, it is preserved by $G$; similarly, $V_2$ is preserved by $G$. We thus see $G=\Aut(V_1) \times \Aut(V_2)$. The claims now follow easily.
\end{proof}

\subsection{Type Ib} \label{ss:Ib}

This class is introduced in the following first-order definition.

\begin{definition} \label{defn:Ib}
A biquadratic space $V$ is of \defn{Type~Ib} if there is some $p \in \bP^1$ such that:
\begin{enumerate}
\item The nullspace $W$ of $\omega(p)$ is non-zero.
\item The total orthogonal complement $U$ of $W$ is the universal homogeneous biquadratic space.
\item We have $V=U \oplus W$.
\end{enumerate}
The \defn{parameter} of $V$ is the pair $(d,p)$ where $d=\dim(W)$. We note that $1 \le d \le \infty$.
\end{definition}

We offer one point of clarification here. ``Homogeneous'' is not obviously a first-order condition. However, ``universal homogeneous'' is equivalent to ``universal f-injective'' (Proposition~\ref{prop:fraisse}), which is easily seen to be a first-order condition.

The value $p$ in the above definition is unique. Indeed, suppose there were some $q \ne p$ such that $\omega(q)$ had a non-zero null vector $x$. Then $x$ would be orthogonal to $W$ under both $\omega(p)$ and $\omega(q)$, and thus belong to $U$, and so $\omega(q)$ would have a null vector on $U$. But this contradicts $U$ being the universal homogeneous space.

\begin{construction} \label{con:Ib}
Let $d \in \{1,2,\ldots,\infty\}$. Let $V_1$ be a $d$-dimensional biquadratic space on which $\omega$ has zero nullspace and $\omega'$ is identically zero, and let $V_2$ be the universal homogeneous biquadratic space. Then $V=V_1 \oplus V_2$ is a Type~Ib space with parameter $(d, \infty)$. Twisting by $\GL(2)$ gives examples with any parameter.
\end{construction}

\begin{proposition}
Two~Ib spaces are isomorphic if and only if their parameters coincide. In particular, a Type~Ib space is linearly $\omega$-categorical.
\end{proposition}

\begin{proof}
It is clear that isomorphic spaces have the same parameter. Let $V$ be a Type~Ib space with parameter $(d, \infty)$, and let $U$ and $W$ be as in Definition~\ref{defn:Ib}. Let $V_1$ and $V_2$ be as in Construction~\ref{con:Ib}. If $x$ belongs to $\ker(\omega \vert_W)$ then $x$ is orthogonal to $W$ under both $\omega$ and $\omega'$, and so $x$ belongs to $U$; since $W \cap U=0$, we find $x=0$. We thus see that $\ker(\omega \vert_W)=0$. It follows that $W$ is isomorphic to $V_1$ as a biquadratic space; of course, $U$ is isomorphic to $V_2$. We thus find that $V$ is isomorphic to $V_1 \oplus V_2$. We therefore see that every Type~Ib space with parameter $(d, \infty)$ is isomorphic to the one from Construction~\ref{con:Ib}. Twisting by $\GL(2)$ gives the general case.
\end{proof}

\begin{proposition}
Let $V$ be a Type~Ib space, and let $G=\Aut(V)$. Then $G$ is linearly oligomorphic, $V$ has length two as a representation of $G$, and $\End_G(V)=k \oplus k$.
\end{proposition}

\begin{proof}
Let $V=U \oplus W$ be a Type~Ib space with parameter $(d, p)$, and let $G=\Aut(V)$. Since $W$ is the nullspace of $\omega(p)$, it is stable by $G$, and since $U$ is the total orthogonal complement of $W$, it too is stable by $G$. Thus each element of $g$ preserves both $U$ and $W$, and preserves both forms on each space. Conversely, any such map is an automorphism of $V$. Thus
\begin{displaymath}
G = \Aut(U) \times \Aut(W).
\end{displaymath}
Since $U$ and $W$ are both irreducible biquadratic spaces, the result follows.
\end{proof}


\subsection{Type Ic} \label{ss:Ic}

This class is introduced in the next first-order definition.

\begin{definition}
A biquadratic space $V$ is of \defn{Type~Ic} if $U=\ker(V)$ is non-zero and $V/U$ is the universal homogeneous biquadratic space. The \defn{parameter} of $V$ is $d=\dim(U)$. We note that $1 \le d \le \infty$.
\end{definition}

As with the Type~Ib definition, this is indeed first-order.

\begin{construction}
Let $d \in \{1,2,\ldots,\infty\}$. Let $V_1$ be a $d$-dimensional biquadratic space where both forms vanish, and let $V_2$ be the universal homogeneous biquadratic space. Then $V=V_1 \oplus V_2$ is a Type~Ic space with parameter $d$. This shows that Type~Ic spaces exist for every parameter $d$.
\end{construction}

\begin{proposition}
Two Type~Ic spaces are isomorphic if and only if their parameters coincide. In particular, a Type~Ic space is linearly $\omega$-categorical.
\end{proposition}

\begin{proof}
Let $V$ and $V'$ be two Type~Ic spaces. It is clear that if $V$ and $V'$ are isomorphic then they have the same parameter. Conversely, suppose that $V$ and $V'$ have the same parameter. Since $U=\ker(V)$ and $U'=\ker(V')$ are vector spaces of the same dimension, we can choose a linear isomorphism $i_0 \colon U \to U'$. Since $V/U$ and $V'/U'$ are both the universal homogeneous biquadratic space, there is an isomorphism $i_1 \colon V/U \to V'/U'$ of biquadratic spaces. Choose a linear map $i \colon V \to V'$ such that $i$ restricts to $i_0$ on $U$, and induces $i_1$ on the quotients. Then $i$ is an isomorphism of biquadratic spaces.
\end{proof}

\begin{proposition}
Let $V$ be a Type~Ic space, and put $G=\Aut(V)$. Then $G$ is linearly oligomorphic, $V$ has length two as a representation of $G$, and $\End_G(V)=k$.
\end{proposition}

\begin{proof}
Put $U=\ker(V)$. Clearly, $U$ is a $G$-subrepresentation of $V$. One easily sees that
\begin{displaymath}
G = \begin{pmatrix} \GL(U) & \Hom(V/U, U) \\ 0 & \Aut(V/U) \end{pmatrix},
\end{displaymath}
where $\Hom$ consists of all linear maps and $\Aut$ consists of automorphisms of biquadratic spaces. The result easily follows from this.
\end{proof}


\section{Type II spaces} \label{s:type2}

\subsection{Type IIa} \label{ss:IIa}

The following definition introduces this class.

\begin{definition} \label{def:IIa}
A biquadratic space $V$ is of \defn{Type~IIa} if it is universal and there exists $p \in \bP^1$ such that:
\begin{enumerate}
\item The kernel $W$ of $\omega(p)$ is non-zero.
\item $W$ is its own total orthogonal complement.
\item For each $q \ne p$ in $\bP^1$ there is an isomorphism $T \colon V/W \to W$ such that $\omega(q)_{T(x)}=\omega(p)_x$ for all $x \in V$. (After fixing representatives for $\omega(q)$ and $\omega(p)$.)
\end{enumerate}
The \defn{parameter} of $V$ is $p$.
\end{definition}

The parameter $p$ is uniquely defined. Indeed, twisting by $\GL(2)$, it suffices to assume the parameter is $\infty$ and show that $\ker(\omega)=0$. Thus suppose $x$ belongs to $\ker(\omega)$. Then $x$ clearly belongs to the total orthogonal complement to $W$, which is assumed to be $W$. However, by (c), we find $\omega'_{T^{-1}(x)}=\omega_x=0$, and so $T^{-1}(x)=0$ in $V/W$, and so $x=0$ since $T$ is an isomorphism.

In light of the above, we can rephrase (c) as follows: for each $q \ne p$ in $\bP^1$ and each $x \in V$ there exists a unique $y \in W$ such that $\omega(q)_y=\omega(p)_x$. This shows that (c) is (equivalent to) a first-order statement, and so the entire Type~IIa condition is first order.

We also note that it suffices to check Definition~\ref{def:IIa}(c) for a single value of $q$. Indeed, for notational simplicity suppose $p=\infty$ and we have checked (c) for $q=0$. Thus $\omega_{T(x)}=\omega'_x$ for all $x \in V$. Since $T(x)$ belongs to $W$, which is the nullspace of $\omega'$, we have $(\omega+q\omega')_{T(x)}=\omega'_x$ for all $x \in V$, which verifies (c) for $q$.

\begin{construction} \label{con:IIa}
Let $V$ be a space with basis $\{e_i, f_i\}_{i \ge 1}$. Define $\omega$ and $\omega'$ on $V$ by the matrices
\begin{displaymath}
\begin{pmatrix} 0 & 1 \\ 1 & 0 \end{pmatrix} \qquad
\begin{pmatrix} 0 & 0 \\ 0 & 1 \end{pmatrix},
\end{displaymath}
where here the $e$ vectors come before the $f$ vectors. Thus, explicitly, we have
\begin{displaymath}
\omega(e_i, f_j) = \delta_{i,j}, \qquad \omega'(f_i, f_j) = \delta_{i,j},
\end{displaymath}
and all other pairings vanish. The kernel $W$ of $\omega'$ is spanned by the $e_i$'s. One readily verifies that it is its own total orthogonal complement. Moreover, the map $T \colon V/W \to W$ defined by $T(f_i)=e_i$ satisfies $\omega_{T(x)}=\omega'_x$ for all $x$. We thus see that $V$ is of Type~IIa with parameter $\infty$. Twisting by $\GL(2)$ realizes any parameter.
\end{construction}

\begin{proposition}
Two Type~IIa spaces with the same parameter are isomorphic. In particular, a Type~IIa space is linearly $\omega$-categorical.
\end{proposition}

\begin{proof}
Let $V$ be a Type~IIa space with parameter $\infty$. We show that $V$ is isomorphic to the space given in Construction~\ref{con:IIa}. Let $W=\ker(\omega')$, and let $T \colon V/W \to W$ be the given isomorphism. The form $\omega'$ is non-degenerate on $V/W$, and so there is an orthonormal basis $\{\ol{y}_i\}_{i \ge 1}$. Let $y_i \in V$ be any lift of $\ol{y}_i$. Let $x_i=T(y_i)$, so that the $x_i$'s form a basis of $W$. Since $W$ is isotropic for $\omega$, we have $\omega(x_i,x_j)=0$. We have
\begin{displaymath}
\omega(x_i, y_j) = \omega(T(y_i), y_j) = \omega'(y_i, y_j) = \delta_{i,j}.
\end{displaymath}
We thus see that the matrices for $\omega$ and $\omega'$ are
\begin{displaymath}
\begin{pmatrix} 0 & 1 \\ 1 & \ast \end{pmatrix} \qquad
\begin{pmatrix} 0 & 0 \\ 0 & 1 \end{pmatrix},
\end{displaymath}
where the first block consists of $x_i$'s, and the second of $y_i$'s.

We now modify the $y_i$ vectors to make the $\ast$ block in $\omega$ vanish. Put
\begin{displaymath}
z_i = y_i - \tfrac{1}{2} \omega(y_i, y_i) x_i - \sum_{j=1}^{i-1} \omega(y_i, y_j) x_j.
\end{displaymath}
The $z_i$'s are a basis for a complementary space to $W$. We have $\omega(x_i,z_i)=\delta_{i,j}$ and $\omega'(z_i,z_j)=\delta_{i,j}$. Moreover, one readily verifies $\omega(z_i,z_j)=0$. We thus see that the matrices for $\omega$ and $\omega'$ in the basis $\{x_i,z_i\}$ match those from Construction~\ref{con:IIa}.
\end{proof}

\begin{proposition}
Let $V$ be a Type~IIa space, and let $G=\Aut(V)$. Then $G$ is linearly oligomorphic, $V$ has length two as a representation of $G$, and $\End_G(V)=k[x]/(x^2)$.
\end{proposition}

\begin{proof}
It suffices to prove this for the space $V$ from Construction~\ref{con:IIa}. Automorphisms of $V$ are given by block matrices of the form
\begin{displaymath}
\begin{pmatrix} A & AB \\ 0 & A \end{pmatrix}
\end{displaymath}
where $A$ is an orthogonal matrix (i.e., $A^t=A^{-1}$) and $B$ is an anti-symmetric matrix. Thus, abstractly, $G$ is the semi-direct product of the infinite orthogonal group with the additive group of anti-symmetric matrices (the Lie algebra of the orthogonal group). It follows that $W$ is the unique subrepresentation of $V$. Both $W$ and $V/W$ are irreducible, as they are the standard representation of the infinite orthogonal group. Thus $V$ has length two. The endomorphism $T$ from Definition~\ref{def:IIa} shows that $\End_G(V)$ contains $k[x]/(x^2)$ as a subalgebra, and this is the whole algebra since $V$ has a unique irreducible subrepresentation.

Finally, let $G_0$ be the subgroup of $G$ consisting of matrices as above with $B=0$. Then $G_0$ is isomorphic to the infinite orthogonal group, and $V$ is a sum of two copies of its standard representation. Since the standard representation of $G_0$ is linearly oligomorphic (Theorem~\ref{thm:oligo}, or direct verification), so is the representation $V$ (Proposition~\ref{prop:sum-oligo}). Since $V$ is linearly oligomorphic for the subgroup $G_0$, it is also linearly oligomorphic for $G$.
\end{proof}

\subsection{Type IIb} \label{ss:IIb}

This class of spaces is introduced in the following first-order definition.

\begin{definition}
A biquadratic space $V$ is of \defn{Type~IIb} if it is universal and there exists $p \in \bP^1$ such that $\ker(\omega(p))$ is non-zero and the total orthogonal complement of $\ker(\omega(p))$ vanishes. The \defn{parameter} of $V$ is $p$.
\end{definition}

Lemma~\ref{lem:IIb-1} below shows that $p$ is the unique value such that $\ker(\omega(p))$ is non-zero, and so the parameter is well-defined.

\begin{construction} \label{con:IIb}
Let $V$ be a space with basis $\{e_i, f_i\}_{i \ge 1}$. Let $B$ be an $\infty \times \infty$ matrix such that every non-trivial linear combination of its columns has infinitely many non-zero entries. We define $\omega$ and $\omega'$ to be the matrices
\begin{displaymath}
\begin{pmatrix} 1 & B^t \\ B & 0 \end{pmatrix} \qquad
\begin{pmatrix} 0 & 0 \\ 0 & 1 \end{pmatrix},
\end{displaymath}
with conventions as in Construction~\ref{con:IIa}. Explicitly,
\begin{displaymath}
\omega(e_i, e_j) = \delta_{i,j}, \qquad
\omega(e_i, f_j) = B_{i,j}, \qquad
\omega'(f_i, f_j) = \delta_{i,j}.
\end{displaymath}
and all other pairings vanish. It is clear that $\ker(\omega')$ is the space spanned by the $e_i$'s. Suppose now that $x \in V$ is orthogonal to the $e_i$'s with respect to $\omega$. Write
\begin{displaymath}
x=\sum_{j=1}^n \alpha_j e_j + \sum_{j=1}^n \beta_j f_j.
\end{displaymath}
Then for $i>n$ we have
\begin{displaymath}
0 = \omega(e_i, x) = \sum_{j=1}^n \beta_j B_{i,j}.
\end{displaymath}
We thus see that the column vector $\sum_{j=1}^n \beta_j B_{\bullet,j}$ has finitely many non-zero entries, and so $\beta_j=0$ for each $j$. Since $0=\omega(e_i, x)=\alpha_i$ for all $i$, we find $x=0$. We thus see that $V$ has Type~IIb with parameter~$\infty$. Twisting by $\GL(2)$ realizes any parameter.
\end{construction}

Our next goal is to prove the following result.

\begin{proposition} \label{prop:IIb-unique}
Two Type~IIb spaces are isomorphic if and only if their parameters agree. In particular, a Type~IIb space is linearly $\omega$-categorical.
\end{proposition}

Twisting by $\GL(2)$, it suffices to treat the case where the parameter is $\infty$. Thus fix a Type~IIb space $V$ with parameter $\infty$. Consider the following category $\cC$:
\begin{itemize}
\item An object is a finite biquadratic space $W$ equipped with a subspace $\fN(W)$ of $\ker(\omega')$.
\item A map $i \colon W \to W'$ is a map of biquadratic spaces such that $i^{-1}(\fN(W'))=\fN(W)$.
\end{itemize}
We regard $V$ as an ind-object of $\cC$, with $\fN(V)=\ker(\omega')$. The zero object in $\cC$ is initial, and so an f-injective of $\cC$ is universal. In particular, any two f-injective objects are isomorphic (Proposition~\ref{prop:fraisse}). The following lemma thus establishes Proposition~\ref{prop:IIb-unique}.

\begin{proposition} \label{prop:IIb-4}
$V$ is an f-injective ind-object of $\cC$.
\end{proposition}

We require a few lemmas before proving this proposition.

\begin{lemma} \label{lem:IIb-1}
If $x,y \in V$ satisfy $\omega_x+\omega'_y=0$ then $x=0$ and $y \in \fN(V)$.
\end{lemma}

\begin{proof}
Evaluating the given equation at $v \in \fN(V)$, we find $\omega(x,v)=0$, i.e., $x$ is $\omega$-orthogonal to $\fN(V)$. Of course, $x$ is also $\omega'$-orthogonal to $\fN(V)$. We thus see that $x$ belongs to the total orthogonal complement of $\fN(V)$, and so $x=0$. The given equation then simplifies to $\omega'_y=0$, which exactly means $y \in \fN(V)$.
\end{proof}

\begin{lemma} \label{lem:IIb-2}
Given any $s \ge 0$, we can find a finite dimensional subspace $E$ of $V$ such that $E \cap \fN(V)=0$, and $\omega$ and $\omega'$ have collective strength at least $s$ on $E$.
\end{lemma}

\begin{proof}
Let $W$ be a biquadratic space of strength $s$ (meaning the two forms have collective strength $s$), and let $e_1, \ldots, e_n$ be a basis of $W$. Define a new biquadratic space $W'$ with basis $e_1, \ldots, e_n, f_1, \ldots, f_n$, as follows. On the $e$ vectors, the forms are as on $W$. We put $\omega'(e_i, f_j)=\delta_{i,j}$, and let all other pairings vanish. Since $V$ is universal, there is an embedding $\sigma \colon W' \to V$. We let $E=\sigma(W)$. It is clear that $E$ has strength $s$. If $x \in W$ is non-zero then $\omega'(x, f_j) \ne 0$ for some $j$, and so $\omega'(\sigma(x), \sigma(f_j)) \ne 0$, which shows that $\sigma(x) \not\in \fN(V)$. Thus $E \cap \fN(V)=0$.
\end{proof}

\begin{lemma} \label{lem:IIb-3}
Let $F$ be a finite dimensional subspace of $V$, let $\lambda_1, \ldots, \lambda_n$ be linearly independent functionals on $V$, and let $s \ge 0$ be given. We can then find a finite dimensional subspace $E$ of $V$ such that
\begin{enumerate}
\item $E \cap (\fN(V)+F)=0$
\item The restriction of the $\lambda_i$'s to $E$ are linearly independent.
\item The forms $\omega$ and $\omega'$ have collective strength at least $s$ on $E$.
\end{enumerate}
\end{lemma}

\begin{proof}
Let $d=\dim(F)$. Applying Lemma~\ref{lem:IIb-2}, let $E'$ be a finite dimensional subspace of $V$ such that $E' \cap \fN(V)=0$, and $\omega$ and $\omega'$ have strength at least $d+s$ on $E'$. We can then find a subspace $E$ of $E'$ of codimension at most $d$ such that $E \cap (\fN(V)+F)=0$. Strength drops by at most $d$ when passing to a codimension $d$ subspace, so the strength of $\omega$ and $\omega'$ on $E$ is at least $s$.

Now, suppose that there is some non-trivial linear combination $\lambda$ of the $\lambda_i$'s that vanishes on $E$. Since $\fN(V)+F+E$ has infinite codimension in $V$, there is some vector $x$ not in this subspace such that $\lambda(x) \ne 0$. We can then replace $E$ with $E+kx$ to obtain a space still satisfying (a) and (c), but where $\lambda$ no longer vanishes. In other words, we have cut down the space of linear relations the $\lambda_i$'s satisfy on $E$. Repeating this procedure finitely many times, we end up with an $E$ also satisfying (b).
\end{proof}

\begin{proof}[Proof of Proposition~\ref{prop:IIb-4}]
Let $W$ and $W'$ be objects of $\cC$, and let $\sigma \colon W \to W'$ and $\tau \colon W \to V$ be embeddings. We must extend $\tau$ to an embedding $\tau' \colon W' \to V$. Without loss of generality, we assume $\dim(W')=\dim(W)+1$. For notational simplicity, we identify $W$ with a subspace of $W'$ via $\sigma$; we also write $\omega$ and $\omega'$ for the forms on all three spaces. Let $\ol{x}_1, \ldots, \ol{x}_n$ be a basis of $W$ such that $\ol{x}_1, \ldots, \ol{x}_m$ forms a basis of $\fN(W)$. Let $x_i=\tau(\ol{x}_i)$. We consider two cases.

\textit{Case 1: $\fN(W)=\fN(W')$.} Let $\ol{y}$ be a vector in $W'$ that does not belong to $W$. Then $\ol{x}_1, \ldots, \ol{x}_n, \ol{y}$ forms a basis for $W'$, and $\ol{x}_1, \ldots, \ol{x}_m$ forms a basis for $\fN(W')$. Applying Lemma~\ref{lem:IIb-3}, let $E$ be a finite dimensional subspace of $V$ such that
\begin{itemize}
\item $E \cap (\fN(V)+\operatorname{span}(x_1, \ldots, x_n))=0$
\item The forms $\omega$ and $\omega'$ have high strength on $E$.
\item The linear forms $\omega(x_i, -)$, for $1 \le i \le n$, and $\omega'(x_i, -)$ for $m<i \le n$, are linearly independent on $E$.
\end{itemize}
Note that the linear forms in the third item above are linearly independent on all of $V$ by Lemma~\ref{lem:IIb-1}. By Corollary~\ref{cor:univ}, we can find $y \in E$ such that
\begin{align*}
\omega(y, y) &= \omega(\ol{y}, \ol{y}), &
\omega(x_i, y) &= \omega(\ol{x}_i, \ol{y}), \quad \text{for $1 \le i \le n$,} \\
\omega'(y, y) &= \omega'(\ol{y}, \ol{y}), &
\omega'(x_i, y) &= \omega'(\ol{x}_i, \ol{y}), \quad \text{for $m<i \le n$.}
\end{align*}
Of course, we also have $\omega'(x_i, y)=\omega'(\ol{x}_i, \ol{y})$ for $1 \le i \le m$, as both sides vanish. We can now define $\tau'$ by putting $\tau'(\ol{y})=y$. Since $y \in E$, we find that $\tau'$ is a morphism in $\cC$.

\textit{Case 2: $\fN(W) \ne \fN(W')$.} Let $\ol{y}$ be a vector in $\fN(W')$ that does not belong to $\fN(W)$. Then $\ol{x}_1, \ldots, \ol{x}_n, \ol{y}$ forms a basis for $W'$, and $\ol{x}_1, \ldots, \ol{x}_m, \ol{y}$ forms a basis for $\fN(W')$. Let $E$ be a finite dimensional subspace of $\fN(V)$ such that
\begin{itemize}
\item $E \cap \operatorname{span}(x_1, \ldots, x_n)=0$.
\item The form $\omega$ has high strength on $E$.
\item The linear forms $\omega(x_i, -)$, for $1 \le i \le n$, are linearly independent on $E$.
\end{itemize}
We now explain how to find $E$. We begin with three observations. First, $\fN(V)$ is infinite dimensional, as otherwise its total orthogonal complement would have finite codimension. Second, the linear forms in the third item are linearly independent on $\fN(V)$ by the argument in Lemma~\ref{lem:IIb-1}. And third, $\ker(\omega \vert_{\fN(V)})=0$, as any element of the kernel belongs to the total orthogonal complement of $\fN(V)$. Now, take a large finite dimensional subspace $\tilde{E}$ of $\fN(V)$ satisfying the second and third conditions. Then the collection of subspaces $E$ of $\tilde{E}$ of a fixed large dimension satisfying the first (or second, or third) condition is a non-empty Zariski open subset of the appropriate Grassmannian. There is thus a single $E$ satisfying all three conditions.

Now, by Corollary~\ref{cor:univ}, we can find $y' \in E$ such that
\begin{displaymath}
\omega(y, y) = \omega(\ol{y}, \ol{y}), \qquad
\omega(x_i, y) = \omega(\ol{x}_i, \ol{y}), \quad \text{$1 \le i \le n$.}
\end{displaymath}
Of course, we also have
\begin{displaymath}
\omega'(y, y) = \omega'(\ol{y}, \ol{y}), \qquad
\omega'(x_i, y) = \omega'(\ol{x}_i, \ol{y}), \quad \text{$1 \le i \le n$}
\end{displaymath}
as both sides of each equation vanish. We can thus define $\tau'$ by putting $\tau'(\ol{y})=y$.
\end{proof}

\begin{remark}
Proposition~\ref{prop:IIb-4} implies that $V$ is a homogenous ind-object of $\cC$. Concretely, this means that if $U$ and $U'$ are finite dimensional subspaces of $V$ and $\sigma \colon U \to U'$ is an isomorphism of biquadratic spaces such that $\sigma$ maps $U \cap \fN(V)$ isomorphically onto $U' \cap \fN(V)$ then there is an automorphism $g$ of $V$ such that $\sigma(x)=gx$ for all $x \in U$. An important point here is that $U \cap \fN(V)$ cannot be determined from the biquadratic space $U$ alone; this is why the objects in $\cC$ are endowed with extra structure.
\end{remark}

We prove one more result:

\begin{proposition} \label{prop:IIb-group}
Let $V$ be a Type~IIb space, and let $G=\Aut(V)$. Then $G$ is linearly oligomorphic, $V$ has length two as a representation of $G$, and $\End_G(V)=k$.
\end{proposition}

Again, twisting by $\GL(2)$ we assume that the parameter of $V$ is $\infty$. We will prove the proposition in a series of lemmas. The first is a slightly stronger version of Proposition~\ref{prop:IIb-4}.

\begin{lemma} \label{lem:IIb-5}
Let $\sigma \colon W \to W'$ and $\tau \colon W \to V$ be as in the proof of Proposition~\ref{prop:IIb-4}. If $F$ is a finite dimensional subspace of $V$ such that $\tau(W) \cap F=0$, then we can extend $\tau$ to $\tau' \colon W' \to V$ such that $\tau'(W') \cap F=0$.
\end{lemma}

\begin{proof}
The proof of Proposition~\ref{prop:IIb-4} applies with slight modification: in Case~1, choose $E$ so
\begin{displaymath}
E \cap (\fN(V)+\operatorname{span}(x_1, \ldots, x_n)+F)=0,
\end{displaymath}
and in Case~2 choose $E$ so
\begin{displaymath}
E \cap (\operatorname{span}(x_1, \ldots, x_n)+F)=0. \qedhere
\end{displaymath}
\end{proof}

The next lemma shows that finite dimensional subspaces of $V$ are definably closed, in the sense of model theory.

\begin{lemma} \label{lem:IIb-6}
Let $W$ be a finite dimensional subspace of $V$, and let $H \subset G$ be the subgroup fixing each vector in $W$. Then $V^H=W$.
\end{lemma}

\begin{proof}
Let $x$ be an element of $V$ that does not belong to $W$, and put $W'=W+kx$. Applying Lemma~\ref{lem:IIb-5} with $\sigma$ and $\tau$ the standard inclusions and $F=kx$, we obtain an embedding $\tau' \colon W' \to V$ that is the identity on $W$, and satisfies $\tau'(x) \ne x$. Of course, we also have the standard embedding $W' \to V$. Since $V$ is homogeneous, there is some $g \in G$ such that $\tau'(v)=gv$ for all $v \in W'$. We have thus produced $g \in H$ such that $gx \ne x$, which completes the proof.
\end{proof}

We now establish the statements in the proposition.

\begin{lemma}
We have $\End_G(V)=k$.
\end{lemma}

\begin{proof}
Suppose $f \colon V \to V$ is a $G$-equivariant map. Let $x \in V$ be non-zero, and let $H \subset G$ be the stabilizer of $x$. Then $V^H=kx$ by Lemma~\ref{lem:IIb-6}. Since $f(V^H) \subset V^H$, we see that $x$ is an eigenvetor of $f$. Since this holds for all $x \in V$, it follows that $f$ is multiplication by a scalar.
\end{proof}

\begin{lemma}
Let $W=\ker(\omega')$. Then $W$ and $V/W$ are irreducible representations of $G$.
\end{lemma}

\begin{proof}
We first show that $W$ is irreducible. Suppose by way of contradiction that $M$ is a proper non-zero subrepresentation of $W$. Let $x$ be a non-zero element of $M$ and let $\lambda$ be a non-zero functional on $W$ that vanishes on $M$. Let $E$ be a finite dimensional subspace of $W$ such that $\omega$ has high strength on $E$ and $\lambda$ is non-zero on $E$. By Corollary~\ref{cor:univ} we can find $y \in E$ such that $\omega(y,y)=\omega(x,x)$ and $\lambda(y)=1$. Since $kx$ and $ky$ are isomorphic objects of $\cC$ (via $x \mapsto y$) and $V$ is homogeneous, there is $g \in G$ such that $gx=y$, and so $y \in M$. But this is a contradiction since $\lambda(y)=1$. Thus $W$ is irreducible.

We now show that $W$ is the unique proper non-zero subrepresentation of $V$. The argument is quite similar. Suppose that $M$ is a proper subrepresentation of $V$ that is not contained in $W$. Let $x$ be an element of $M$ that does not belong to $W$, and let $\lambda$ be a non-zero functional on $V$ that vanishes on $M$. Applying Lemma~\ref{lem:IIb-3}, let $E$ be a finite dimensional subspace of $V$ such that $E \cap W=0$, the forms $\omega$ and $\omega'$ have high strength on $E$, and $\lambda \ne 0$ on $E$. By Corollary~\ref{cor:univ}, we can find $y \in E$ such that
\begin{displaymath}
\omega(y,y)=\omega(x,x), \qquad \omega'(y,y)=\omega'(x,x), \qquad \lambda(y)=1.
\end{displaymath}
Since $kx$ and $ky$ are isomorphic objects of $\cC$ (via $x \mapsto y$) and $V$ is homogeneous, there is $g \in G$ such that $gx=y$, and so $y \in M$. But this is a contradiction since $\lambda(y)=1$. We thus see that $W$ is the unique non-zero proper subrepresentation of $V$. In particular, $V/W$ is irreducible.
\end{proof}

Define an object $E_{n,m,\ell}$ of $\cC$ as follows:
\begin{itemize}
\item The space has basis $x_1, \ldots, x_n, y_1, \ldots, y_m, z_1, \ldots, z_{\ell}$.
\item The $x_i$'s and $y_i$'s are orthonormal under $\omega$, while the $z_i$'s span the kernel of $\omega$.
\item The $z_i$'s are orthonormal under $\omega'$, while the $x_i$'s and $y_i$'s span the kernel of $\omega'$.
\item The space $\fN(E_{n,m,\ell})$ is spanned by the $x_i$'s.
\end{itemize}

\begin{lemma} \label{lem:IIb-7}
Every $n$-dimensional space in $\cC$ embeds into $E_{2n,2n,2n}$.
\end{lemma}

\begin{proof}
Let $W$ be an $n$-dimensional object of $\cC$. Let $w_1, \ldots, w_n$ be a basis of $W$ such that $w_1, \ldots, w_r$ is a basis of $\fN(W)$. Choose an embedding $\alpha \colon W \to E_{2n,0,0}$ that is compatible with $\omega$. (Here we are using the fact that any $n$-dimensional quadratic space embeds into the non-degenerate quadratic space of dimension $2n$.) Let $y'_{r+1}, \ldots, y'_n$ span an $n-r$ dimensional isotropic subspace of $E_{0,2n,0}$, and define a linear map $\beta \colon W \to E_{0,2n,0}$ by $\beta(w_i)=0$ for $1 \le i \le r$ and $\beta(w_i)=y'_i$ for $r<i\le n$. Finally, let $\gamma \colon W/\fN(W) \to E_{0,0,2n}$ be an embedding compatible with $\omega'$. Now let $\delta \colon W \to E_{2n,2n,2n}$ be the sum of $\alpha$, $\beta$, and $\gamma$. One readily verifies that $\delta$ is compatible with $\omega$ and $\omega'$, and satisfies $\delta^{-1}(\fN(E_{2n,2n,2n}))=\fN(W)$, as required.
\end{proof}

\begin{lemma}
The group $G$ is linearly oligomorphic.
\end{lemma}

\begin{proof}
Let $E=E_{2n,2n,2n}$ and fix an embedding $E \subset V$, which is possible since $V$ is a universal ind-object of $\cC$. If $W \subset V$ is $n$-dimensional then there is an embedding $W \to E$ in $\cC$ (Lemma~\ref{lem:IIb-7}). Since $V$ is homogeneous, it follows that there is $g \in G$ such that $gW \subset E$. We thus see that the map $\Gr_n(E) \to \Gr_n(V)/G$ is surjective, as required.
\end{proof}

\subsection{Type IIc} \label{ss:IIc}

This class is introduced in the following first-order definition.

\begin{definition}
A biquadratic space $V$ is of \defn{Type~IIc} if it is universal and there exists $p \in \bP^1$ such that:
\begin{enumerate}
\item The nullspace of $\omega(p)$ is non-zero.
\item The total orthogonal complement of $\ker(\omega(p))$ is itself.
\item We have $\omega_x+\omega'_y=0$ only if $x$ and $y$ are linearly dependent.
\item We have $\ker(V)=0$.
\end{enumerate}
The \defn{parameter} of $V$ is $p$.
\end{definition}

While the definition looks quite similar to the Type~IIa case, these spaces are more similar to those of Type~IIb. We note that the parameter $p$ is unique. Indeed, suppose $p=\infty$ for simplicity, and let us show that $\omega$ has zero nullspace. Let $x \in \ker(\omega)$. Then $x$ belongs to the total orthogonal complement of $\ker(\omega')$, which is $\ker(\omega')$. Thus $x \in \ker(V)$, which vanishes.

\begin{construction} \label{con:IIc}
Let $V$ be a space with basis $\{e_i, f_i\}_{i \ge 1}$. Let $B$ be an $\infty \times \infty$ matrix such that every non-trivial linear combination of its columns has infinitely many non-zero entries. We define $\omega$ and $\omega'$ on $V$ by
\begin{displaymath}
\begin{pmatrix} 0 & B^t \\ B & 0 \end{pmatrix} \qquad
\begin{pmatrix} 0 & 0 \\ 0 & 1 \end{pmatrix},
\end{displaymath}
using our usual conventions. The nullspace $W$ of $\omega'$ is the span of the $e_i$'s. One readily verifies that $V$ is Type~IIc with parameter $\infty$. Twisting by $\GL(2)$ realizes all parameter values.
\end{construction}

Our next goal is to prove the following result.

\begin{proposition} \label{prop:IIc-unique}
Two Type~IIc spaces are isomorphic if and only if their parameters agree. In particular, a Type~IIc space is linearly $\omega$-categorical.
\end{proposition}

Twisting by $\GL(2)$, it suffices to treat the case where the parameter is $\infty$. Thus fix a Type~IIc space $V$ with parameter $\infty$. Consider the following category $\cC$:
\begin{itemize}
\item An object is a finite biquadratic space $W$ equipped with a subspace $\fN(W)$ of $\ker(\omega')$ that is isotropic for $\omega$.
\item A morphism $i \colon W \to W'$ is a map of biquadratic spaces such that $i^{-1}(\fN(W'))=\fN(W)$.
\end{itemize}
We regard $V$ as an ind-object of $\cC$, with $\fN(V)=\ker(\omega')$. As in the Type~IIb case, the above proposition follows from the subsequent one.

\begin{proposition}
$V$ is an f-injective ind-object of $\cC$.
\end{proposition}

\begin{proof}
We first note that Lemmas~\ref{lem:IIb-1},~\ref{lem:IIb-2}, and~\ref{lem:IIb-3} extend without difficulty to the present case. We now follow the proof of Proposition~\ref{prop:IIb-4} from the Type~IIb case. Let $W$, $W'$, $\sigma$, $\tau$, $\ol{x}_i$, and $x_i$ be as in that proof. We again proceed in two cases. The first case, where $\fN(W)=\fN(W')$ follows the exact argument from before.

We now consider the second case, where $\fN(W) \ne \fN(W')$. Let $\ol{y}$ be a vector in $\fN(W')$ that does not belong to $\fN(W)$. Then $\ol{x}_1, \ldots, \ol{x}_n, \ol{y}$ forms a basis for $W'$, and $\ol{x}_1, \ldots, \ol{x}_m, \ol{y}$ forms a basis for $\fN(W')$. Let $E$ be a finite dimensional subspace of $\fN(V)$ such that
\begin{itemize}
\item $E \cap \operatorname{span}(x_1, \ldots, x_n)=0$.
\item The linear forms $\omega(x_i, -)$, for $m < i \le n$, are linearly independent on $E$.
\end{itemize}
Note that if $\omega_x$ restricts to~0 on $\fN(V)$ then $x$ belongs to the total orthogonal complement of $\fN(V)$, which is $\fN(V)$. Thus the forms appearing in the second item above are linearly independent on $\fN(V)$. We also note that $\fN(V)$ is infinite dimensional, as otherwise its total orthogonal complement, which is itself would have finite codimension, and so $\omega'$ would have finite rank, contradicting universality. It follows that such a space $E$ exists. Now, choose $y \in E$ such that
\begin{displaymath}
\omega(x_i, y) = \omega(\ol{x}_i, \ol{y}) \quad \text{for $m < i \le n$.}
\end{displaymath}
Of course, we also have
\begin{align*}
\omega(y, y) &= \omega(\ol{y}, \ol{y}) &
\omega(x_i, y) &= \omega(\ol{x}_i, \ol{y}) \quad \text{for $1 \le i \le m$} \\
\omega'(y,y) &= \omega'(\ol{y}, \ol{y}) &
\omega'(x_i, y) &= \omega'(\ol{x}_i, \ol{y}) \quad \text{for $1 \le i \le n$}
\end{align*}
as everything above vanishes. We can thus define $\tau'$ by putting $\tau'(\ol{y})=y$.
\end{proof}

We require one more result:

\begin{proposition}
Let $V$ be a Type~IIc space, and let $G=\Aut(V)$. Then $G$ is linearly oligomorphic, $V$ has length two as a representation of $G$, and $\End_G(V)=k$.
\end{proposition}

\begin{proof}
The proof is similar to that of Proposition~\ref{prop:IIb-group}.
\end{proof}

\section{Type III spaces} \label{s:III}

The following definition introduces our final class of biquadratic spaces.

\begin{definition} \label{defn:III}
A biquadratic space $V$ is of \defn{Type~III} if there exists a decomposition $V=U \oplus U'$ and a linear isomorphism $T \colon U \to U'$ with the following properties:
\begin{enumerate}
\item For $x,y \in V$ we have $\omega_x=\omega'_y$ if and only if $x \in U$ and $y=T(x)$.
\item The space $U$ equipped with the four quadratic forms
\begin{displaymath}
\eta_1 = \omega \vert_U, \qquad
\eta_2 = \omega' \vert_U, \qquad
\eta_3 = T^*(\omega \vert_{U'}), \qquad
\eta_4 = T^*(\omega' \vert_{U'})
\end{displaymath}
is an irreducible multi-quadratic space.
\end{enumerate}
We define the \defn{parameter} of $V$ to be $E=\Theta(\eta_1, \ldots, \eta_4)$; recall that this is the subspace of $k^4$ consisting of points $\alpha$ such that $\sum_{i=1}^4 \alpha_i \eta_i=0$.
\end{definition}

We note that this is a first-order definition. Indeed, the spaces $U$ and $U'$ are definable, as is the function $T$ (meaning its graph is definable). For instance, $x \in U$ if and only if there exists $y \in V$ such that $\omega_x=\omega'_y$; this is a first-order statement for membership in $U$. Additionally, condition (b) is first-order by our characterization of irreducible multi-quadratic spaces (Theorem~\ref{thm:mqclass}).

To study Type~III spaces, we introduce the following category $\cD$:
\begin{itemize}
\item An object is a biquadratic space $V$ equipped with a direct sum decomposition $V=V_1 \oplus V_2$ and an isomorphism $T \colon V_1 \to V_2$ of vector spaces such that $\omega_x=\omega'_{Tx}$ for all $x \in V_1$.
\item A morphism $V \to W$ is an embedding of biquadratic spaces that is compatible with $T$ and maps $V_i$ into $W_i$.
\end{itemize}
Note that a Type~III space canonically carries the structure of an object of $\cD$. Let $\cE$ be the category of 4-quadratic spaces. We define functors
\begin{displaymath}
\Phi \colon \cD \to \cE, \qquad \Psi \colon \cE \to \cD
\end{displaymath}
as follows. First suppose that $V$ is an object of $\cD$. We define $\Phi(V)$ to be the space $V_1$ equipped with the four quadratic forms
\begin{displaymath}
\eta_1=\omega \vert_{V_1}, \qquad \eta_2=\omega' \vert_{V_1}, \qquad \eta_3=T^*(\omega \vert_{V_2}), \qquad \eta_4=T^*(\omega' \vert_{V_2}).
\end{displaymath}
The definition of morphisms in $\cD$ ensures that $\Phi$ is a functor. Now suppose that $U$ is an object of $\cE$, with forms $\{\eta_i\}_{1 \le i \le 4}$. We define $V=\Psi(U)$ as follows. As a vector space, $V=V_1 \oplus V_2$, where each $V_i$ is a copy of $U$. The map $T \colon V_1 \to V_2$ is the identity. For $x,y \in V_1$, we define
\begin{displaymath}
\omega(x+Ty,x+Ty) = \eta_1(x,x) + \eta_3(y, y) + 2 \eta_4(x, y)
\end{displaymath}
\begin{displaymath}
\omega'(x+Ty,x+Ty) = 2\eta_1(x, y) + \eta_2(x,x) + \eta_4(y, y)
\end{displaymath}
One readily verifies that $V$ is naturally an object of $\cD$, and that $\Psi$ is naturally a functor. Moreover, the compositions $\Psi \circ \Phi$ and $\Phi \circ \Psi$ are both naturally isomorphic to the identity functors. Thus $\Phi$ and $\Psi$ are quasi-inverse equivalences.

\begin{proposition}
Two Type~III spaces are isomorphic if and only if their parameters coincide. In particular, a Type~III space is linearly $\omega$-categorical.
\end{proposition}

\begin{proof}
Let $V$ and $V'$ be Type~III spaces with the same parameter. Both $\Phi(V)$ and $\Phi(V')$ are irreducible 4-quadratic spaces. Since $V$ and $V'$ have the same parameter, these two 4-quadratic spaces are isomorphic (Corollary~\ref{cor:mqclass-1}). Since $\Phi$ is an equivalence, it follows that $V$ and $V'$ are isomorphic.
\end{proof}

\begin{proposition}
Let $V$ be a Type~III space, and let $G=\Aut(V)$. Then $G$ is linearly oligomorphic, $V$ has length two as a representation of $G$, and $\End_G(V)=\rM_2(k)$.
\end{proposition}

\begin{proof}
Since $V_1$ and $V_2$ are canonically defined in terms of $\omega$ and $\omega'$, they are stable by $G$. Similarly, since $T \colon V_1 \to V_2$ is canonically defined, it is $G$-equivariant. Thus $V_1$ and $V_2$ are isomorphic $G$-representations. By definition of Type~III, $V_1$ is an irreducible representation. Thus $V$ has length two, and $\End_G(V)=\rM_2(k)$.

By the above reasoning, $G$ is the automorphism group of $V$ regarded as an object of $\cD$. Thus $G$ is the automorphism group of the 4-quadratic space $\Phi(V)=V_1$. It follows from Corollary~\ref{cor:mqclass-1} that the action of $G$ on $V_1$ is linearly oligomorphic. Since $V$ is isomorphic to a direct sum of two copies of $V_1$, the action of $G$ on $V$ is also linearly oligomorphic (Proposition~\ref{prop:sum-oligo}).
\end{proof}

To close our analysis of Type~III spaces, we turn to the question of existence. To this end, we say that a subspace $E \subset k^4$ is \defn{good} if whenever $(a, c, 0, b)$ and $(c, 0, b, a)$ both belong to $E$ we have $a=b=c=0$.

\begin{proposition}
If $V$ is a Type~III space then its parameter $E$ is good. Conversely, every good $E$ arises from a Type~III space.
\end{proposition}

\begin{proof}
Let $U$ be an irreducible 4-quadratic space, let $E$ be defined as above, and let $V=\Psi(U)$. We analyze the condition Definition~\ref{defn:III}(a). For $x,x',y,y' \in V_1$, we have
\begin{align*}
\omega(x+Ty, x'+Ty') &= \eta_1(x,x')+\eta_3(y,y')+\eta_4(x,y')+\eta_4(y,x') \\
\omega'(x+Ty, x'+Ty') &= \eta_1(x,y')+\eta_1(y,x')+\eta_2(x,x')+\eta_4(y,y').
\end{align*}
We thus see that $\omega(x_1+Ty_1, -)=\omega'(x_2+Ty_2,-)$ if and only if
\begin{align*}
\eta_1(x_1,-)+\eta_4(y_1,-) &= \eta_1(y_2,-)+\eta_2(x_2,-), \\
\eta_3(y_1,-)+\eta_4(x_1,-) &= \eta_1(x_2,-)+\eta_4(y_2,-).
\end{align*}
This happens if and only if there is some $v \in V_1$ such that
\begin{displaymath}
(x_1,y_1,x_2,y_2)=(a,b,c,d) v
\end{displaymath}
for some $a,b,c,d \in k$ and
\begin{displaymath}
(a-d,-c,0,b), \qquad (-c, 0, b, a-d)
\end{displaymath}
belong to $E$. If $E$ is good, this forces $a=d$ and $b=c=0$, meaning $y_1=x_2=0$ and $y_2=x_1$, which shows that $V$ is Type~III. Conversely, if $E$ is not good then there exists a tuple $(a,b,c,d) \in k^4$ such that at least one of $a-d$, $b$, or $c$ is non-zero and both of the above vectors belong to $E$. We then find $\omega(av+bTv,-)=\omega'(cv+dTv,-)$, which shows that $V$ is not Type~III. This completes the proof.
\end{proof}

\section{The main theorem} \label{s:mainthm}

We now prove Theorem~\ref{mainthm}. We note that, while the introduction assumed the field $k$ had characteristic~0, we now simply assume the characteristic is not~2 (as well as the other conditions imposed in \S \ref{s:multiq}). Let $V$ be a length two universal biquadratic space. We must show that $V$ belongs to one of the seven families appearing in Table~\ref{tab:main}. In what follows, ``representation'' will always mean a representation of $G=\Aut(V)$.

\begin{lemma}
If $\ker(V)$ is non-zero then $V$ has Type~Ic.
\end{lemma}

\begin{proof}
Put $U=\ker(V)$. We cannot have $U=V$, as then $\omega$ and $\omega'$ would vanish identically and $V$ would not be universal. Thus $U$ is a non-zero proper subrepresentation. Since $V$ has length two, it follows that $V/U$ is an irreducible representation of $G$. The forms $\omega$ and $\omega'$ descend to $V/U$, and the action of $G$ on $V/U$ respects these forms. Thus $V/U$ is an irreducible biquadratic space. If some $\omega(p)$ vanishes on $V/U$, then it would vanish on all of $V$, and $V$ would not be universal; thus this cannot happen. Therefore, by the classification of irreducible biquadratic spaces (Theorem~\ref{thm:mqclass}), $V/U$ is the universal homogeneous biquadratic space. We thus see that $V$ has Type~Ic, as required.
\end{proof}

We assume in what follows that $\ker(V)$ vanishes. Let $\Delta=\Delta(\omega, \omega')$. Recall that this is the set of pairs $(x,y) \in V \oplus V$ such that $\omega_x=\omega'_y$. It is clear that $\Delta$ is a subrepresentation of $V \oplus V$. We let $\pi_i \colon \Delta \to V$ be the projection map. We have $\ker(\pi_1)=\ker(\omega')$ and $\ker(\pi_2)=\ker(\omega)$, where we identify $V \oplus 0$ with $V$.

\begin{lemma} \label{lem:main-1}
$\Delta$ has length one or two.
\end{lemma}

\begin{proof}
If $\Delta=0$ then $V$ is the universal homogeneous biquadratic space (Theorem~\ref{thm:mqchar}), which is irreducible (Proposition~\ref{prop:uh-irred}). Thus $\Delta$ must have length at least one.

Suppose now that $\Delta$ has length at least three. Since $V$ has length two, the projection maps $\pi_i$ must have non-zero kernels. Thus $W=\ker(\omega)$ and $W'=\ker(\omega')$ are non-zero. We cannot have $W=V$ or $W'=V$, for then $V$ would not be universal. Thus $W$ and $W'$ are irreducible subrepresentations of $V$. We have already made the assumption $W \cap W'=0$, and so $V=W \oplus W'$. Since the kernel of $\pi_1$ has length one, the image of $\pi_1$ has length at least two, and so $\pi_1$ is surjective. This means that for every $x \in V$ there is $y \in V$ such that $\omega_x=\omega'_y$. Since $\omega_x$ vanishes on $W$ and $\omega'_y$ vanishes on $W'$, we find that $\omega_x$ is identically~0. Since $x$ is arbitrary, we have $\omega=0$, which contradicts the universality of $V$. We conclude that $\Delta$ has length at most two.
\end{proof}

\begin{lemma}
If $\Delta$ has length two then $V$ is of Type~Ia or IIa.
\end{lemma}

\begin{proof}
We proceed in three cases.

\textit{(a) There exist $p \ne q \in \bP^1$ such that both $\omega(p)$ and $\omega(q)$ have non-zero nullspace.} Twisting by $\GL(2)$, we can assume that $\omega$ and $\omega'$ themselves have non-zero nullspaces $W$ and $W'$. As in the proof of Lemma~\ref{lem:main-1}, $W$ and $W'$ are irreducible and $V=W \oplus W'$. Since $\omega$ has infinite rank, $W'$ must be infinite dimensional; similarly, $W$ must be infinite dimensional. Thus $V$ is of Type~Ia.

\textit{(b) The form $\omega(p)$ has non-zero nullspace for a unique $p \in \bP^1$.} Twisting by $\GL(2)$, we assume $p=0$, i.e. $\omega$ has a non-zero nullspace $W$. Since $V$ is universal, $W \ne V$, and so $W$ is irreducible. Since $\ker(\omega')=0$, it follows that $\pi_1$ is injective, and is therefore an isomorphism since $\Delta$ and $V$ have the same length. The projection $\pi_2$ has kernel $W \oplus 0$, and so its image is an irreducible subrepresentation $U$ of $V$. We thus see that for every $x \in U$ there is a $y \in V$ such that $\omega_y=\omega'_x$; moreover, $y$ is unique modulo $W$. In other words, we have an isomorphism $T \colon U \to V/W$ such that $\omega'_x = \omega_{Tx}$ for $x \in U$.

We claim $U=W$. Suppose not. Since $U$ and $W$ are irreducible, it follows that $U \cap W=0$, and so $V=U \oplus W$. We thus have $V/W=U$ canonically, and so we can regard $T$ as an endomorphism of $U$. By Schur's lemma (Proposition~\ref{prop:schur}), it is thus multiplication by some scalar $\alpha$, and so we have $\omega'_x=\alpha \omega_x$ for $x \in U$. But this means that $\omega'-\alpha \omega$ has nullspace $U$, a contradiction. This establishes the claim.

Thus $T \colon W \to V/W$ is an isomorphism. For $v \in W$, we see that $\omega'_v = \omega_{Tv}$ vanishes on $W$. This shows that $W$ is isotropic for $\omega'$. Thus the total orthogonal complement $W'$ of $W$ contains $W$, and so we have $W'=W$, as $W'=V$ is impossible (this would mean $\ker(V)$ contains $W$). We therefore find that $V$ is of Type~IIa.

\textit{(c) The form $\omega(p)$ has vanishing nullspace for all $p \in \bP^1$.} In this case, $\pi_i$ has no kernel, and is therefore an isomorphism since its source and target have the same length. It follows that there is a unique isomorphism $T \colon V \to V$ such that $\omega_x=\omega'_{Tx}$ for all $x \in V$. Since $\End_G(V)$ is finite dimensional (Corollary~\ref{cor:schur1}), $T$ satisfies a minimal polynomial, and so $V$ decomposes into generalized eigenspaces for $T$. However, if $x$ is an $\alpha$-eigenvector of $T$ then $\omega_x=\alpha \omega'_x$, and so $x$ belongs to the nullspace of $\omega-\alpha \omega'$, and so $x=0$. Thus all eigenspaces for $T$ vanish, which is a contradiction. We conclude that this case does not occur.
\end{proof}

\begin{lemma}
If $\Delta$ is irreducible then $V$ is of Type~Ib, IIb, IIc, or III.
\end{lemma}

\begin{proof}
We consider two cases.

\textit{(a) Some $\omega(p)$ has non-zero nullspace.} Twisting by $\GL(2)$, we assume that $\omega$ has non-zero nullspace $W$. Since $\Delta$ is irreducible and contains $W \oplus 0$, we have $\Delta=W \oplus 0$. In particular, $\omega'$ has zero nullspace. Let $U$ be the orthogonal complement of $W$ under $\omega'$, i.e., $U$ consists of those vectors $x \in V$ such that $\omega'_x \vert_W=0$. We cannot have $U=V$, for then $W$ would be contained in the nullspace of $\omega'$. If $U=0$ then $V$ has Type~IIb. Thus suppose that $U$ is irreducible.

First suppose $U \ne W$. Then $V = U \oplus W$. This is an orthogonal direct sum for both $\omega$ and $\omega'$. On $W$, the form $\omega$ vanishes, while $\omega'$ is non-degenerate. The biquadratic space $W'$ is universal homogeneous. Thus $V$ has Type~Ib.

Now suppose $U=W$. This means that $W$ is an isotropic subspace for $\omega'$. Thus the total orthogonal complement $W'$ of $W$ contains $W$, and we must have $W'=W$ since we cannot have $W'=V$. Thus $V$ has Type~IIc.

\textit{(b) Every $\omega(p)$ has zero nullspace.} The two projection maps $\Delta \to V$ are both injective, and so their images are irreducible subrepresentations $W_1$ and $W_2$ of $V$. Moreover, there is a unique isomorphism $T \colon W_1 \to W_2$  such that $\omega_v=\omega'_{Tv}$ for $v \in W_1$. The space $\Delta$ is the graph of $T$. We cannot have $W_1=W_2$, for then $T$ would be multiplication by a scalar $\alpha$ and $\omega-\alpha \omega'$ would have nullspace $W_1$. Thus $V=W_1 \oplus W_2$. We therefore see that $V$ canonically carries the structure of an object of the category $\cD$ defined in \S \ref{s:III}. The equivalence with $\cE$ studied there shows that condition Definition~\ref{defn:III}(b) holds, and so $V$ is Type~III.
\end{proof}

\end{document}